\documentclass[11pt]{amsart}

\usepackage{comment}

\usepackage{dsfont}
\usepackage{amsmath}
\usepackage{amssymb}
\usepackage{graphicx}

\usepackage{amssymb}

\usepackage{amsfonts}

\usepackage{soul}

\usepackage{mathpazo}
\usepackage{color}
\usepackage{yfonts}

\usepackage{paralist}

\usepackage{stmaryrd}

\usepackage{amsxtra}

\def\cU{\mathcal U}

\def\cA{\mathcal A}

\newcommand{\cL}{\mathcal L}

\newcommand{\norm}[1]{\| #1\|}

\def\cA{          \mathcal A}
\def\cD{          \mathcal D}

\let\cal\mathcal

\def \R{{\mathbb R}}

\def \N{{\mathbb N}}

\newcommand{\T}{{\mathbb T}}
\newcommand{\prf}{{\begin{proof}}}
\newcommand{\epf}{{\end{proof}}}

\newcommand{\PP}{{\mathbb P}}
\newcommand{\CC}{{\mathbf C}}

\def\bA{          \mathbf A}

\def\cA{          \mathcal A}
\def\cB{          \mathcal B}
\def\cC{          \mathcal C}
\def\cD{          \mathcal D}
\def\EV{          \mathbb E}

\def\S{          \mathbb{S}}

\def\PF{          \mathcal P}

\DeclareMathOperator{\diff}{Diff}

\newtheorem{prop}{\sc Proposition}

\newtheorem{lemma}{\sc lemma}
\newtheorem{cor}{\sc corollary}

\theoremstyle{definition}

\def\bee{\begin{equation}}
\def\eee{\end{equation}}

\newtheorem{defi}{\sc Definition}

\theoremstyle{rema}

\newtheorem{rema}{\sc Remark}

\newcommand{\pdvr}[2]
{\dfrac{\partial^{#2} #1}{\partial \theta^{#2_1} \partial r^{#2_2}}}
\newcommand{\pdvrs}[2]
{\partial^{#2} #1 /\partial \theta^{#2_1} \partial r^{#2_2}}
\newcommand{\ary}{\begin{eqnarray}}
\newcommand{\eary}{\end{eqnarray}}

\newcommand{\aryst}{\begin{eqnarray*}}
\newcommand{\earyst}{\end{eqnarray*}}

\newcommand{\enmt}{\begin{enumerate}}
\newcommand{\eenmt}{\end{enumerate}}

\newtheorem{thm}{Theorem}

\numberwithin{equation}{section}

\author{Zhiyuan Zhang}
\begin{document}

\title[On differentiability of SRB measures]{On the smooth dependence of SRB measures for partially hyperbolic systems}

\date{\today}
\maketitle

\begin{abstract}
In this paper, we study the differentiability of SRB measures for partially hyperbolic systems.

We show that for any $s \geq 1$, for any integer $\ell \geq 2$, any sufficiently large $r$, any $\varphi \in C^{r}(\T, \R)$ such that the map $f : \T^2 \to \T^2, f(x,y) = (\ell x, y + \varphi(x))$ is $C^r-$stably ergodic, there exists an open neighbourhood of $f$ in $C^r(\T^2,\T^2)$ such that any map in this neighbourhood has a unique SRB measure with $C^{s-1}$ density, which depends on the dynamics in a $C^s$ fashion.

We also construct a $C^{\infty}$ mostly contracting partially hyperbolic diffeomorphism  $f: \T^3 \to \T^3$ such that all $f'$ in a $C^2$ open neighbourhood of $f$ possess a unique SRB measure $\mu_{f'}$ and the map $f' \mapsto \mu_{f'}$ is strictly H\"older at $f$, in particular, non-differentiable. This gives a partial answer to Dolgopyat's Question 13.3 in \cite{Do1}.

\end{abstract}

\tableofcontents
\addtocontents{toc}{\protect\setcounter{tocdepth}{1}}

\section{Introduction}

There is a lot of interest in understanding the ergodic aspect of partially hyperbolic systems. For conservative dynamics, one of the fundamental questions is proving ergodicity. In this direction, we have stable ergodicity conjecture which attempts to describe the generic picture of volume preserving partially hyperbolic systems. For non-conservative dynamics, one tries to describe the dynamics through studying distinguished invariant measures. A prominent role is played by SRB measures. 
\begin{defi}
For any $C^1$ diffeomorphism $f: X \to X$ on a compact Riemannian manifold $X$, a probability measure $\mu$ on $X$ is called a SRB measure for $f$ if there exists a subset $Y(\mu) \subset X$ of positive Lebesgue measure such that for any $x \in Y(\mu)$, any continuous function $\phi$ on $X$, $\frac{1}{n}\sum_{i=0}^{n-1}\phi(f^{i}(x))$ converges to $\int \phi d\mu$ as $n$ tends to infinity.
\end{defi}
A satisfactory understanding of SRB measures for generic dynamics is currently lacking, despite of having some deep results in several models, see \cite{AlBoVi, BeYo, Do1, RoRoTaUr, Ts2} just to list a few. 

For partially hyperbolic systems, the existence of SRB measures is proved for several cases: 1. mostly expanding dynamics in \cite{AlBoVi}; 2. mostly contracting dynamics in \cite{BoVi, Do1}; 3. generically for partially hyperbolic surface endomorphisms in \cite{Ts2}. Known uniqueness result of SRB measures, for example in \cite{Do1, RoRoTaUr}, usually assume some form of transitivity. An even more refine question is the differentiability of SRB measures.
In \cite{Do1}, it is shown that for partially hyperbolic, dynamically coherent, u-convergent mostly contracting $f$ on a three-dimensional manifold, there is a unique SRB measure $\nu_f$. If in addition that $f$ is also stably dynamically coherent, then $f$ is stably mostly contracting, and the SRB measure is known to exhibit H\"older dependence on the dynamics. In \cite{Do1} Question 13.3, Dolgopyat asked whether or not for mostly contracting dynamics $f$, the map $f \mapsto \nu_f$ is actually smooth ? We refer the readers to \cite{DeLi, DoViYa} for recent advances in the study of mostly contracting dynamics.

The question of the differentiability of SRB measures had been previously studied by several authors. It has its roots in statistical physics, and has applications in averaging theory and the removability of zero Lyapunov exponents.
The differentiability of SRB measures were previously known for Axiom A diffeomorphisms by \cite{Ru}. For a class of rapidly mixing, partially hyperbolic systems with isometric center dynamics, the differentiability is proved by Dolgopyat in \cite{Do2}. On the other hand, to the best of our knowledge, the non-differentiability of SRB measures ( when the existence and uniqueness is proved ) is unknown for partially hyperbolic systems, despite of having some speculations ( see Problem 4 in \cite{ChDo} ). In fact, the breakdown of the differentiability is poorly understood for multidimensional dynamics in general. For one-dimensional dynamics, Whitney-H\"older dependence is proved for a family of smooth unimodal maps in \cite{BaBeSc}, with matching upper and lower bounds for the H\"older exponents. For more results on the nondifferentiability of SRB measures for one-dimensional dynamics, we refer the reader to the references in \cite{Ba}. We mention that in \cite{Ba}, the study of the breakdown of the differentiability of SRB measures for higher dimensional dynamics was proposed as a future research direction.

One of the purpose of this paper is to prove the existence, uniqueness and differentiability of SRB measures for perturbations of a class of area-preserving endomorphisms which are special cases of those studied in \cite{Do3}. We mention a recent work \cite{Ga} on a similar class of systems. We note that in contrast to \cite{Do1, RoRoTaUr}, our method does not directly use any form of transitivity for the map in question.
  On the other hand, we give a method of constructing partially hyperbolic diffeomorphisms and endomorphisms at which the set of uGibbs states ( see Definition \ref{PartII2ugibbssrb} and the footnote ) is not differentiable. We can also require our diffeomorphism to be mostly contracting satsfying the conditions in Theorem II \cite{Do1}, which is known to imply the uniqueness of SRB measure/ uGibbs state. This gives a partial answer to Question 13.3 in \cite{Do1} : we have an example at which linear response breaks down, but we know no non-trivial example of mostly contracting system where linear response holds.
  Moreover by Theorem I in \cite{Do1}, the mostly contracting diffeomorphism we contruct is exponentially mixing with respect to the unique SRB measure, for H\"older observables.  On the other hand, we mention that linear response can appear for slowly mixing systems, see \cite{BaTo}.

\section{Main results}

\begin{defi}
Let $M$ be a compact Riemannian manifold.
Given integers $r \geq s \geq 1$, and an open set $\cal V \subset C^r(M, M)$. We say that $\{f_t\}_{t \in (-1,1)}$ is a $C^s$ family in $\cal V$ through $f_0$, if $f_t \in \cal V$ for any $t \in (-1,1)$, and 
\aryst
\norm{\{f_t\}_{t \in (-1,1)}}_{s,r} := \sup_{0 \leq i \leq s, 0 \leq j \leq r, (t,x) \in I \times M} \norm{\partial_{t}^{i} \partial_x^{j}f_t(x)} < \infty
\earyst

Given any integers $r \geq r' \geq 2$ and an open set $\cU \subset C^r(M,M)$. Assume that for each $f \in \cU$ there exists a unique SRB measure $\mu_f$. Then we say that $f \mapsto \mu_f$ is $C^{r'}$ restricted to $\cU$, if for any $C^{r}$ family $\{f_t\}_{t \in (-1,1)}$ in $\cU$ through $f$, for any $\phi \in C^{r}(M)$, the map $t \mapsto \int \phi d\mu_{f_t}$ is $C^{r'}$ at $t =0$.
\end{defi}

We will prove the existence, uniqueness and differentiability of SRB measures for endomorphisms close to a class of skew-products which we now define. 

For any integers $r \geq 2$, $\ell \geq 2$,  any $\varphi \in C^r( \T ,\R)$, we define a  $C^{r}$ map $f : \T^2 \to \T^2$ by $f(x,y) = (\ell x, y + \varphi(x)), \forall (x,y)\in \T^2$.
We denote by $\cal U^{rot}_{\ell, r}$ the set of $C^r$ maps defined as above for all $\varphi \in C^r(\T,\R)$. We say that $f$ is $C^r-$stably ergodic in $ \cU^{rot}_{\ell, r}$ if all $f' \in  \cU^{rot}_{\ell, r}$ in a $C^r$ open neighbourhood of $f$ are ergodic.

\begin{thm}\label{PartII2thm diff u-G}
For each $r \geq 20$, $1 \leq r' \leq \frac{r}{2} - 9$, $\ell \geq 2$, for any $f_{rot}\in \cU^{rot}_{\ell, r}$ that is $C^r-$stably ergodic in $\cal U^{rot}_{\ell,r}$, there is a $C^{r}$ open neighbourhood of $f_{rot}$ in $C^{r}(\T^2, \T^2)$, denoted by $\cal U$, such that the following is true. Any $f \in \cU$ admits a unique SRB measure $\mu_{f'}$ having $C^{r'-1}$ density, and $f \mapsto \mu_f$ is $C^{r'}$ restricted to $\cU$.

\end{thm}

By Theorem 3.4 in \cite{Do3}, we know that the set of maps  in $\cU^{rot}_{\ell, r}$ that is $C^r-$stably ergodic in $\cal U^{rot}_{\ell,r}$ form a $C^r$ open and dense subset of $\cU^{rot}_{\ell, r}$. It is obvious that our theorem does not extend to nonergodic $f^{rot}$, so in this aspect our theorem is optimal.
 By Theorem 3.3 in \cite{Do3}, for maps in $\cU^{rot}_{\ell, r}$, being $C^r-$stably ergodic in $\cal U^{rot}_{\ell,r}$ is equivalent to being  infinitesimally non-integrable, defined in \cite{Do3}.

Our method for proving Theorem \ref{PartII2thm diff u-G} is based on the work of Tsujii in his study of decay estimates. Our new input emphasis on using higher regularity and the weak perturbation theory of transfer operators in \cite{GoLi, KeLi}. We believe our method for proving the uniqueness of SRB would be of independent interest.

Our next result is on the nondifferentiability of SRB measures. As we mentioned above, the existence of SRB measure in general is already difficult. So in order to state our theorem in a more general context, we recall the following more general notion. 
\begin{defi} \label{PartII2ugibbssrb}
Let $f: X \to X$ be a $C^2$ partially hyperbolic system on a compact Riemannian manifold $X$. We denote by $uGibbs(f)$ the set of $f-$invariant Borel probability measure $\mu \in \cal M(X)$ such that $\mu$ has absolutely continuous conditional measures on unstable manifolds. \footnote{In some places this notion is also called SRB measure. In our paper, we reserve the term SRB measure for those with a basin of positive Lebesgue measure. }
\end{defi}
We will establish examples of mostly contracting partially hyperbolic systems stably having a unique SRB measure, while the SRB measures depend on the dynamics in a strictly H\"older fashion.  We can even make the Holder exponent to be arbitrarily small. 

\begin{thm}\label{PartII2thm nondiff}
For any $r=2,3, \cdots, \infty$, for any $\theta \in (0,1)$, there is a $C^{r}$ partially hyperbolic diffeomorphism ( resp. endomorphism ) $f : X \to X$ on a compact Riemannian manifold $X$ such that the following is true. There is a $C^{r}$ family $\{f_t\}_{t \in (-1,1)}$ in the space of $C^{r}$ partially hyperbolic diffeomorphisms ( resp. endomorphisms ) through $f$, and a $C^{r}$ function $\phi : X \to \R$ such that for any  $\{\mu_t \in uGibbs(f_t)\}_{t \in (-1,1)}$, the function $t \mapsto \int \phi d\mu_t$ is not $\theta-$Holder at $t=0$. 
Moreover we can choose $f$ to satisfy Theorem II in \cite{Do1}, that is, $f$ can be a stably dynamically coherent, u-convergent, mostly contracting map on $\T^3$.
\end{thm}

The notion u-convergent in Theorem \ref{PartII2thm nondiff} is defined in \cite{Do1} for 3D partially hyperbolic systems $f$ as follows. We say $f$ is u-convergent if for any $\varepsilon > 0$, there exists an integer $n > 0$ such that for any two unstable manifolds of length between $1$ and $2$, denoted by $V_1, V_2$, there exists $x_j \in V_j, j=1,2$ such that $d(f^n(x_1), f^{n}(x_2)) < \varepsilon$.

Our Theorem \ref{PartII2thm nondiff} give an example to Dolgopyat's Question 13.3 in \cite{Do1}. An interesting aspect of our construction is that this nondifferentiability comes with some form of stability. See Further Aspect 2.

\subsection*{Further Aspect}

1. We will later see that we can choose $f$ in Theorem \ref{PartII2thm nondiff} so that $\inf_{f' \in \cU^{rot}_{\ell, r}}d_{C^0}(f, f')$ can be made arbitrarily small, and to exhibit lack of transversality. Theorem \ref{PartII2thm diff u-G}, \ref{PartII2thm nondiff} as stated does not exclude the possible existence of a region where the SRB measures are differentiable at a generic map, and are non-differentiable at the others ( on a nonempty set ). We think it is very likely that there exists a nonperturbative $C^r$ open neighbourhood of $\cU^{rot}_{\ell, r}$ with such property. Indeed, we think some form of transversality condition would be necessary for the differentiability of SRB measures. There are other works that explore the relation between transversality and ( fractional ) linear response, for example \cite{BaBeSc, BaKuLu, LiSm}

2. The non-differentiable example we constructed is a skew product, and is stable under sufficiently localised perturbation preserving the skew product ( See Corollary \ref{PartII2stablenondiff} ). It would be interesting to construct an open set of diffeomorphisms where the non-differentiability of SRB measures hold.

\subsection*{Plan of the paper} 
We will recall Tsujii's transversality condition in Section \ref{PartII2Transversality property}, and reduce the proof Theorem \ref{PartII2thm diff u-G} to Proposition \ref{PartII2prop diff}, which we prove in Section \ref{PartII2Spectral gap in Anisotropic Banach space}. In Section \ref{PartII2Nondifferentiability of u-Gibbs states}, we give precise conditions for the construction and verify these conditions in Subsection \ref{PartII2Conditions for the construction} and finish the proof of Theorem \ref{PartII2thm nondiff} in Subsection \ref{PartII2Proving nondifferentiability}.

\section{Transversality property}\label{PartII2Transversality property}

The proof of Theorem \ref{PartII2thm diff u-G} is divided into two parts using a transversality condition due to Tsujii in \cite{Ts2, Ts}, which we now introduce.
\begin{defi} \label{PartII2cone and transv}
For any $\alpha > 0$, we set
\aryst
 \CC(\alpha) = \{(x,y) \in \R^2 | |y| \leq \alpha |x| \}.
\earyst
More generally, for any line $L \subset \R^2$ containing the origin, any $\beta > 0$, we denote
\aryst
\CC(L, \beta) = \{(x,y) \in \R^2 \setminus \{0\} | \angle( (x,y), L) \leq \beta \} \bigcup \{0\}.
\earyst
Given $\ell \geq 2$, $\gamma_0 \in (\ell^{-1}, 1)$ and $\theta > 0$.
Denote $\CC_0 = \CC(\theta)$. Then for any $f \in \cU^{rot}_{\ell, r}$ written as
$f(x,y) = (\ell x, y + \varphi(x))$ such that 
\ary \label{PartII2gamma0theaf}
(\gamma_0 \ell - 1)\theta > \norm{D\varphi},
\eary
 we have that $\CC_0$ is strictly invariant under $Df$ in the sense that
\ary \label{PartII2cone strictly invariant}
Df_{z}(\CC_0) \Subset \CC(\gamma_0\theta) \mbox{ for any $z \in \T^2$}.
\eary
Here and after, for two cones $\CC, \CC' \subset \R^2$, we denote $\CC \Subset \CC'$ if the closure of $\CC$ is contained in the interior of $\CC'$ except for the origin. For any cone $\CC$, we set
\aryst
\CC^{*} = \{u \in \R^2 | \exists v \in \CC \mbox{ such that }\langle u,v \rangle = 0 \}.
\earyst

Given any $\ell \geq 2, \gamma_0 \in (\ell^{-1}, 1), \theta > 0, f$ satisfying \eqref{PartII2cone strictly invariant}, for any $z \in \T^2 $, any $n \geq 1$, any $w_1,w_2 \in f^{-n}(z)$, we say that $w_1 \pitchfork w_2$ if
\aryst
Df^{n}_{w_1}(\CC_0) \bigcap Df^{n}_{w_2}(\CC_0) = \{0\}
\earyst
otherwise we say $w_1 \not\pitchfork w_2$.
We define
\aryst
m(f, n) &=& \sup_{z \in \T^2}\sup_{w \in f^{-n}(z)} \ell^{-n} \#\{ \zeta \in f^{-n}(z) | \zeta \not\pitchfork w\}  \leq 1, \\
m(f) &=& \limsup_{n \to \infty} m(f, n)^{\frac{1}{n}} \leq 1.
\earyst
\end{defi}
By \eqref{PartII2cone strictly invariant}, it is direct to see that 
\ary \label{PartII2m is less that its finite scale approx}
m(f) \leq m(f, n)^{\frac{1}{n}}, \quad \forall n \geq 1.
\eary
Then we have the following easy but important consequence,
\begin{center} 
\textit{The function $f \mapsto m(f)$ is  upper semicontinuous in $C^1$ topology.}
\end{center}

Using the exponent $m(f)$, the proof of Theorem \ref{PartII2thm diff u-G} splits into two parts.
\begin{prop}\label{PartII2prop diff}
Given any integers $r  \geq 20$, $1 \leq r' \leq \frac{r}{2}-9$, $\ell \geq 2$. For any $\gamma_0 \in (\ell^{-1},1), \theta > 0$, $f \in \cal U^{rot}_{\ell, r}$ satisfying \eqref{PartII2gamma0theaf} and  $m(f) < 1$, there exists an $C^r$ open neighbourhood of $f$ in $C^{r}(\T^2, \T^2)$, denoted by $\cal U$, such that any $f' \in \cal U$ admits a unique SRB measure $\mu_{f'}$ having $C^{r'-1}$ density, and $f' \mapsto \mu_{f'}$ is $C^{r'}$ restricted to $\cal U$.
\end{prop}
\begin{prop}\label{PartII2prop density}
For any integers $r \geq 1, \ell \geq 2$,  any $f \in \cU^{rot}_{\ell, r}$ that is $C^r-$stably ergodic in $\cU^{rot}_{\ell, r}$, there exist $\gamma_0 \in (\ell^{-1},1), \theta > 0$ satisfying \eqref{PartII2gamma0theaf} and $m(f) < 1$.
\end{prop}
\begin{proof}
The proof is very similar to Theorem 1.4 in \cite{Ts}.
We denote $f(x,y) = (\ell x, y + \varphi(x)), \forall (x,y) \in \T^2$ and choose any $\gamma_0 \in (\ell^{-1}, 1), \theta > 0$ such that \eqref{PartII2gamma0theaf} is true.
If $m(f) = 1$, then for any $n \geq 1$, there exists $z_n \in \T^2$ such that for any $w,w' \in f^{-n}(z_n)$, $Df^{n}_{w}(\CC_0) \bigcap Df^{n}_{w'}(\CC_0) \neq \emptyset$. Thus there exists a line in $\R^2$, denoted by $L_n$ contained in $\CC_0$, such that $Df^{n}_{\omega}(\CC_0) \subset \CC(L_n, C\ell^{-n})$ for all $w \in f^{-n}(z_n)$ and some constant $C$ independent of $n$. After passing to a subsequence, we can assume that $z_n \to z$, $L_n \to L$. We let $W$ be the set of $(z',L') \in\T^2 \times \mathbb P(\R^2)$ such that for any $n \geq 0$, any $w' \in f^{-n}(z')$, $Df^{n}_{w'}(\CC_0) \subset \CC(L', C\ell^{-n})$. We easily verify that $W$ is closed and completely invariant. Moreover, $(z,L) \in W$. This shows that for any $z \in \T^2$ there exists $\Psi(z) \in \mathbb P(\R^2)$ such that $(z, \Psi(z)) \in W$. It is easy to see that the choice of $\Psi(z)$ is unique and depends only on the first coordinate of $z$. Let $\psi : \T \to \R$ be a function such that $\Psi(z) = [ \R (1,\psi(x)) ], \forall z = (x,y) \in \T^2$. Then we have
\aryst
\ell^{-1}(\psi(x) + \varphi'(x)) = \psi(\ell x), \quad \forall x \in \T.
\earyst
Then for any two sequences $(y_n)_{n \geq 0},(y'_n)_{n \geq 0}$ in $\T$ such that 
$\ell y_{n+1} = y_n, \ell y'_{n+1} = y'_n$ and $y_0 = y'_0$, we have
\aryst
\sum_{i \geq 1} l^{-i} \varphi'(y_i) = \sum_{i \geq 1} l^{-i} \varphi'(y'_i) 
\earyst
But this shows that $f$ does not satisfy the infinitesimal completely non-integrability condition in Section 3.2 \cite{Do3}. We then conclude the proof by Theorem 3.3 in \cite{Do3}.
\end{proof}

\begin{proof}[Proof of Theorem \ref{PartII2thm diff u-G}: ] 
Our theorem follows immediately by combining Proposition \ref{PartII2prop diff} and Proposition \ref{PartII2prop density}.
\end{proof}

 We will prove Proposition \ref{PartII2prop diff} in Section \ref{PartII2Spectral gap in Anisotropic Banach space}.

\section{Spectral gap in Anisotropic Banach space}\label{PartII2Spectral gap in Anisotropic Banach space}

Our strategy for proving Proposition \ref{PartII2prop diff} is the following. We construct Anisotropic Sobolev spaces $W_{\Theta, p, q}$ following Tsujii in \cite{Ts}. Different from \cite{Ts}, we consider positive $p,q$, which corresponds to smaller and smoother spaces. We will consider a filtration of such spaces, and establish Lasota-Yorke's inequalities for Perron-Frobenius operator $\PF$ acting on these spaces. These give us control of the essential spectrums of $\PF$. Such control is ultimately due to our hypothesis that transversality strongly dominates the possible contraction in the center space. We then use a general theorem of Gou\"ezel-Liverani in \cite{GoLi} to show the differentiability result.

 Throughout this section, we will need to study inequalities associated to $f^n$ for $f \in C^r(\T^2, \T^2)$ and for different $n$'s. We use $C$ to denote positive constants which are independent of $n$, and use $C_n$ to denote positive constants which may depend on $n$. Constants $C, C_n$ are uniform in a $C^r$ open neighbourhood of $f$, and may vary from line to line.
\subsection{Anisotropic Sobolev spaces}

In this section, we will collection some basic notions from \cite{Ts}.
Throughout this section, 
we denote $R = (-\frac{1}{4}, \frac{1}{4})^2$ and $Q = (-\frac{1}{3}, \frac{1}{3})^2$.

We say $\Theta$ is a polarisation if it  is a combination $\Theta = (\CC_{+}, \CC_{-}, \varphi_{+}, \varphi_{-})$ of closed cones $\CC_{\pm}$ in $\R^2$ and $C^{\infty}$ functions $\varphi_{\pm} : \S^1 \to [0,1]$ on the unit circle $\S^1 \subset \R^2$ satisfying $\CC_{+} \bigcap \CC_{-} = \{0\}$ and
\aryst
\varphi_{+} =  \begin{cases}1, \mbox{ if } \xi \in \S^1 \bigcap \CC_{+}, \\ 0, \mbox{ if }\xi \in \S^1 \bigcap \CC_{-} \end{cases} , \quad
 \varphi_{-} = 1 - \varphi_{+}
\earyst
For two polarisation $\Theta=(\CC_{+}, \CC_{-}, \varphi_{+}, \varphi_{-})$ and $\Theta'=(\CC'_{+}, \CC'_{-}, \varphi'_{+}, \varphi'_{-})$, we write $\Theta < \Theta'$ if $\R^2 \setminus \CC'_{+} \Subset \CC_{-}$.

For a $C^{\infty}$ function $\chi : \R \to [0,1]$ satisfying 
$
\chi(s) = \begin{cases} 1, \mbox{ for } s \leq 1 \\ 0, \mbox{ for } s \geq 2 \end{cases}.
$
For a polarisation $\Theta = (\CC_{+}, \CC_{-}, \varphi_{+}, \varphi_{-})$, an integer $n \geq 0$, and $\sigma \in \{+,-\}$, we define $C^{\infty}$ function $\psi_{\Theta, n, \sigma} : \R^2 \to [0,1]$ by 
\aryst
\psi_{\Theta, n, \sigma}(\zeta) = \begin{cases} \varphi_{\sigma}(\zeta/ |\zeta|)\cdot(\chi(2^{-n}|\zeta|) - \chi(2^{-n+1}|\zeta|) ), \quad n \geq 1 \\ \chi(|\zeta|)/2, \quad n = 0 \end{cases}
\earyst

For a function $u \in L^2(R)$, we denote the Fourier modes by
\aryst 
\cal F(u)(\zeta) = \int e^{-2\pi iy \cdot \zeta} u(y) dy, \quad \zeta \in \R^2
\earyst
and define
\aryst
u_{\Theta, n, \sigma}(x) = \psi_{\Theta, n,\sigma}(D)u(x) :=  \int e^{2\pi ix \cdot \zeta} \psi_{\Theta, n,\sigma}(\zeta) \cal F(u)(\zeta)  d\zeta
\earyst

For any open set $X \in \R^2$, any $r \in (0, \infty]$, we denote by $C^{r}_0(X)$ the set of compactly supported $C^r$ functions on $X$.
For any $p \in \R$, for any $u \in C^{\infty}_{0}(\R^2)$, we denote its Sobolev norm $\norm{u}_{H^{p}}$ by
\aryst
\norm{u}_{H^p} = (\int (|\zeta|^{2}+1)^p |\cal F(u)(\zeta)|^2 d\zeta)^{\frac{1}{2}}
\earyst
It is well-known that for $p \in \N$, 
\ary \label{PartII2sobolev norm equiv}
\norm{u}_{H^p}^2 \sim \sum_{j=0}^{p} \norm{D^ju}_{L^2}^2
\eary
For an open set $X \subset \R^2$, we denote by $H^p_0(X)$ the completion of $C^{\infty}_0(X)$ with respect to $\norm{\cdot}_{H^p}$.

For a polarisation $\Theta = (\CC_{+}, \CC_{-}, \varphi_{+}, \varphi_{-})$ and a real number $p$, we define the semi-norms $\norm{\cdot}^{+}_{\Theta,p}$ and $\norm{\cdot}^{-}_{\Theta,  q }$ on $C^{\infty}_0(R)$ by
\aryst
\norm{u}^{\sigma}_{\Theta,  c(\sigma) } = (\sum_{n \geq 0} 2^{2c(\sigma)n}\norm{u_{\Theta, n,\sigma}}_{L^2}^2)^{1/2},
\earyst
where we set $c(+) = p$ and  $c(-) = q$.

We define the anisotropic Sobolev norm $\norm{\cdot}_{\Theta, p, q}$ on $C^{\infty}_0(R)$ for real numbers $p$ and $q$ by 
\aryst
\norm{u}_{\Theta, p,q} = ((\norm{u}^{+}_{\Theta, p})^2 + (\norm{u}^{-}_{\Theta,q})^2)^{1/2}
\earyst
For any $p,q \in \R$, any polarisation $\Theta$, we denote by $W_{\Theta, p,q}(R)$ the completion of $C^{\infty}_0(R)$ with respect to the norm $\norm{\cdot}_{\Theta, p,q}$. 

In the following two lemmata, we collect some basic properties of anisotropic Sobolev norms. 
\begin{lemma}\label{PartII2lem norm basics}
For any $0 \leq p' < p, 0 \leq q' < q$ satisfying $p'  \geq  q', p  \geq  q$, any polarisations $\Theta' < \Theta$, we have
\enmt
\item[$(1)$] $C^{p}_0(R) \subset H^{p}_0(R) \subset W_{\Theta, p,q}(R) \subset H^q_0(R)$. If $q \geq 2$, then $W_{\Theta, p,q}(R) \subset C^{q-2}(\overline{R})$,
\item[$(2)$] $W_{\Theta, p,q}(R) \subset W_{\Theta', p,q}(R)$,
\item[$(3)$] We have a compact inclusion $W_{\Theta, p,q}(R) \subset W_{\Theta, p',q'}(R)$.
\eenmt
\end{lemma}
\begin{proof}
The first 3 inclusions in (1) and (2) are obvious. The inclusion $W_{\Theta, p,q}(R) \subset C^{q-2}(\overline{R})$ for $q \geq 2$ follows from $W_{\Theta, p,q}(R) \subset H^q_0(R)$ and Sobolev's embedding theorem. For (3), we refer the reader to Proposition 5.1 in \cite{BaTs}.
 \end{proof}

\begin{lemma}\label{PartII2lem norm unit partition}
Let $r \geq 1$ and let $g_i : \R^2 \to [0,1], 1 \leq i \leq I$, be a family of functions, $C^r$ in the interior of $R$, and satisfy $\sum_{i=1}^{I} g_i(x) \leq 1$ for $x \in R$. Let $\Theta$ and $\Theta'$ be polarisations such that $\Theta' < \Theta$, and let $1 \leq q \leq p \leq r$ be integers. Then  for all $u \in C_0^{r}(R)$  we have
\aryst 
(\sum_{i=1}^{I} \norm{g_i u}_{\Theta',p,q}^2)^{\frac{1}{2}} \leq C \norm{u}_{\Theta, p, q} + C' \norm{u}_{\Theta, p-1, q-1}
\earyst
where $C$ does not depend on  $\{g_i\}$, while $C'$ may. 
Further, if $\sum_{i=1}^{I} g_i(x) \equiv 1$ for all $x \in R$ in addition, then for all $u \in C_0^{r}(R)$ we have
\aryst
\norm{u}_{\Theta',p,q} \leq \nu (\sum_{i=1}^{I} \norm{g_i u}_{\Theta,p,q}^2)^{\frac{1}{2}} + C'\sum_{i=1}^{I} \norm{g_i u}_{\Theta, p-1, q-1}
\earyst
where $\nu$ is the intersection multiplicity of the supports of the functions $g_i$ for $1 \leq i \leq I$.
\end{lemma}
\begin{proof}
This is a more general case of Lemma 2.3 in \cite{Ts}. The proof follows from straightforward adaptions. The first inequality is essentially proved in Appendix C \cite{Ts}, the only difference being that instead of $\norm{g_i u}_{L^{2}} \leq \norm{u}_{L^2}$, we use 
\aryst 
\norm{g_iu}_{H^q} \leq \norm{g_iD^{q}u}_{L^2} + C(g_i) \norm{u}_{H^{q-1}}
\earyst
The second inequality is essentially proved in Lemma 7.1 \cite{BaTs}.
\end{proof}

To exploit the expansion in the unstable direction, we consider the following situation. Let $r \geq 2, \rho \in C^{r-1}_0(R)$ be supported inside an open set $U \subset R $ and let $S: U \to S(U) \subset R$ be a $C^r$ diffeomorphism. Consider operator
$L : C^{r-1}(R) \to C^{r-1}(R)$ defined by
\aryst
Lu(x) = \begin{cases} \rho(x) u( S(x)), \quad \forall x \in U \\ 0, \quad \mbox{otherwise} \end{cases}
\earyst
Assume that for polarisations $\Theta = (\CC_{\pm}, \varphi_{\pm}), \Theta' = (\CC'_{\pm}, \varphi'_{\pm})$, we have
\aryst
(DS_{\zeta})^{tr}(\R^2 \setminus \CC_{+}) \Subset \CC'_{-}, \quad \forall \zeta \in U
\earyst
where $(DS_{\zeta})^{tr}$ denotes the transpose of $DS_{\zeta}$. Put
\aryst
 \gamma(S) &=& \min_{\zeta \in U}|\det DS_{\zeta}|\\
\Lambda(S, \Theta') &=& \sup \{\frac{\norm{(DS_{\zeta})^{tr}(v)}}{\norm{v}}| \zeta \in U, (DS_{\zeta})^{tr}(v) \notin \CC'_{-}\}
\earyst
The following is essentially contained in the proof of Lemma 2.4 in \cite{Ts}. 
We refer the readers to the Appendix for the details

\begin{lemma}\label{PartII2lemma ly-ineq local}  Given integers $r \geq 7, 0 \leq q \leq p < \frac{r}{2}-3$. Then the operator $L$ extends boundedly to $L:  W_{\Theta, p,q}(R) \to W_{\Theta', p,q}(R)$. If in addition $q \geq 1$, then we have for $u \in W_{\Theta, p,q}(R)$ that 
\aryst
\norm{Lu}^{-}_{\Theta', q} &\leq& C\norm{\rho}_{L^{\infty}}\gamma(S)^{-\frac{1}{2}}\norm{DS}^q  \norm{u}_{\Theta, p,q} + C' \norm{u}_{\Theta, p-1, q-1} \\
\norm{Lu}^{+}_{\Theta', p} &\leq& C\norm{\rho}_{\infty} \gamma(S)^{-\frac{1}{2}} \Lambda(S, \Theta')^p \norm{u}_{\Theta, p,q} + C' \norm{u}_{\Theta, p-1, q-1} 
\earyst
here constant $C$ does not dependent on $\Theta, \Theta', S, \rho$ while $C'$ may.
\end{lemma}

For any $p,q \in \R$, any polarisation $\Theta$, we define a norm $\norm{\cdot}_{\Theta,p,q}$ for $C^{\infty}(\T^2)$ in the following way. We construct a finite collection of translations of $R$ in $\T^2$, defined by $\{R_{a} := \kappa_a(R) \}_{\alpha \in A}$, where $A$ is a finite set in $\T^2$ and $\kappa_a : Q \to \T^2$ is the embedding defined by $\kappa_a (z) = z + a, \forall z \in Q$. Let $R_a = \kappa_a(R)$ and $Q_a = \kappa_a(Q)$. We assume that $\T^2 \subset \bigcup_{a \in A}R_a$. We choose a unit partition $\{\rho_a \in C^{\infty}(\T^2, [0,1]) \}_{a \in A}$ such that
\aryst
 \sum_{a \in A} \rho_{a} \equiv 1, \quad supp(\rho_a) \subset R_a, \forall a \in A.
\earyst
For each $u \in C^{\infty}(\T^2)$, we define 
\aryst
\norm{u}_{\Theta,p,q} = (\sum_{a \in A} \norm{(\rho_{a} u) \circ \kappa_a}_{\Theta,p,q}^2 )^{\frac{1}{2}}
\earyst
and we let $W_{\Theta, p, q}(\T^2)$ be the completion of $C^{\infty}(\T^2)$ with respect to $\norm{\cdot}_{\Theta,p,q}$.

\begin{rema}
The construction of anisotropic Banach spaces adapted to dynamically systems was originally due to Baladi and Tsujii in \cite{BaTs}, and then used by Tsujii in \cite{Ts} to study a class of suspension semi-flows. Similar ideas also appeared in \cite{AvGoTs}. In their papers, the dynamics are either uniformly hyperbolic, or have natural invariant measures, so they only studied the case where $q \leq 0 < p$ in order to be able to prove decay for rough observables. We need to consider $0 < q < p$ in order to prove our uniqueness of SRB measure.
\end{rema}

\subsection{Transfer operators and Lasota-Yorke's inequality}
In the rest of this section, we let $r \geq r' \geq 2$, $\ell \geq 2$ and assume that $f$ is $C^r$ close to $\cU^{rot}_{\ell, r}$.
It is a classical fact and easy to verify that the density $\rho$ (w.r.t. the Lebesgue measure ) of any absolute continuous $f-$invariant measure $\mu$ is a fixed point of the Perron-Frobenius operator $ \PF_f : L^{1}(\T^2) \to L^1(\T^2)$ associated to $f$, defined by,
\aryst
\PF_f u(z)  = \sum_{w \in f^{-1}(z)} u(w) \det(Df(w))^{-1}, 
\earyst
Moreover, we have for any $u,v \in L^2(\T^2)$ that
\ary \label{PartII2eq pf operator}
(\PF_f u, v )_{L^2} = (u, v \circ f)_{L^2}.
\eary
In the following, we briefly denote $\PF = \PF_f$.

We define for any $n \in \N$, any $a,b \in A$, any $u \in C_0^{r-1}(R)$ that
\aryst
P^n_{a,b}u (x) = \rho_a \kappa_a (x) \sum_{\kappa_b(y) \in f^{-n}( \kappa_a (x))} u(y) \det(Df^{n}(\kappa_{b}(y)))^{-1}
\earyst

Then for any $0 \leq p,q \leq r-1$, any polarisation $\Theta$, any $u \in C^{r-1}(\T^2)$, we have
\ary \label{PartII2normpfuthetapq2}
\norm{\PF^n u}_{\Theta, p,q}^2 &=& \sum_{a \in A}\norm{\sum_{b \in A}(\rho_{a} \PF^n ( \rho_{b} u)) \circ \kappa_a}_{\Theta, p, q}^2  \nonumber \\
&\leq&C \sum_{a \in A}\sum_{b \in A}\norm{(\rho_{a} \PF^n ( \rho_{b} u)) \circ \kappa_a}_{\Theta,p,q}^2  \nonumber \\
&=&C \sum_{a \in A}\sum_{b \in A} \norm{ P^n_{a,b} ((\rho_{b} u ) \circ \kappa_b )}_{\Theta, p,q}^2
\eary

We fix any constants  $\gamma_0 \in (\ell^{-1}, 1), \theta > 0$ such that \eqref{PartII2cone strictly invariant} is satisfied for $f, \CC_0 = \CC(\theta)$ and $\gamma_0$. This is true if, for example, when $f \in \cU^{rot}_{\ell, r}$ and \eqref{PartII2gamma0theaf} is satisfied. In the following, $m(f)$ is defined using cone $\CC_0$.

Let $\check{\Theta}, \Theta, \Theta', \hat{\Theta}$ be polarisations denoted by
\aryst
\check{\Theta} = (\check{\CC}_{\pm}, \check{\varphi}_{\pm}),  \quad
\Theta' = (\CC'_{\pm}, \varphi'_{\pm}), \quad \Theta = (\CC_{\pm}, \varphi_{\pm}), \quad \hat{\Theta} = (\hat{\CC}_{\pm}, \hat{\varphi}_{\pm})
\earyst
such that
\ary \label{PartII2checkhat}
\check{\Theta} < \Theta' <  \Theta  < \hat{\Theta}
\eary
and
\ary
 \label{PartII2check cc cc hat cc relation 2}
(\CC(\gamma_0 \theta))^{*} \Subset \hat{\CC}_{-} &\Subset& (\R^2 \setminus \check{\CC}_{+}) \Subset \CC_0^{*}, \\
(Df^{-1}_{z})^{tr}(\R^2 \setminus \check{\CC}_{+}) &\Subset& \hat{\CC}_{-}, \quad \forall z \in \T^2  \label{PartII2check cc cc hat cc relation 3}
\eary
Moreover, we always assume that $\R(0,1)$ is contained in the interior of $\hat{\CC}_{-}$.
Such choice is possible since by \eqref{PartII2cone strictly invariant}, 
\ary \label{PartII2dfz-1trdfzsub}
(Df_z^{-1})^{tr}(\CC_0^{*}) = ((Df_z)(\CC_0))^{*}\Subset (\CC(\gamma_0 \theta))^{*}, \quad \forall z \in \T^2
\eary

In the following, we fix $\hat{\Theta}, \Theta', \Theta, \check{\Theta}$. For any $h \in (0, \log \ell)$, integer $N_0 > 0$, we let $\cU^{h, N_0}$ be the set of $C^r$ covering maps $g: \T^2 \to \T^2$ of  degree $\ell$ satisfying \eqref{PartII2cone strictly invariant} and
\enmt
\item $\norm{g^{N_0}}_{C^r} < \ell^{N_0}e^{N_0h}$,
\item $\ell^{N_0}e^{\frac{N_0h}{2}} >  \frac{\norm{Dg_z^{N_0}(v)}}{\norm{v}} > \ell^{N_0}e^{-\frac{N_0h}{2}}, \forall  z \in \T^2, v \neq 0, v \in \CC_0$,
\item $e^{\frac{N_0h}{2}} > \frac{\norm{Dg_z^{N_0}(v)}}{\norm{v}} > e^{-\frac{N_0h}{2}}, \forall z \in \T^2, v \neq 0, Dg^{N_0}_z(v) \notin \CC_0$,
\item $(Dg_z^{N_0})^{tr}(v) \in \R^2 \setminus \hat{\CC}_{-} \mbox{ and }\frac{\norm{(Dg_z^{N_0})^{tr}(v)}}{\norm{v}} > \ell^{N_0}e^{-N_0h}, \forall  z \in \T^2, v \neq 0, v \in \R^2 \setminus \hat{\CC}_{-}$.
\eenmt

It is straightforward to check that for any $f^{rot} \in \cU^{rot}_{\ell, r}$, for any $h > 0$, there exists $N_0 = N_0(f^{rot}, h) > 0$ such that $\cU^{h, N_0}$ contains an $C^r$ open neighbourhood of $f^{rot}$ in $C^r(\T^2, \T^2)$.

By \eqref{PartII2cone strictly invariant}, for any $f \in \cU^{h, N_0}$, $N > N_0$, denote a local inverse branch of $f^N$ denoted by $H : U \to \T^2$, i.e. $f^{N}H = Id |_{U}$, we have
\aryst
DH_z(\R^2 \setminus \CC_0) \Subset ( \R^2 \setminus \CC_0), \quad \forall z \in U
\earyst
Then for any $N > N_0$ we have 
\ary  \label{PartII2term 1}
\det(Df^{N}_z) \geq C\ell^{N}e^{-Nh}, \quad \forall z \in \T^2
\eary
and for all  $H$ as above, we have
\ary  \label{PartII2term 2}
\norm{DH_{z}} \leq C e^{N h}, \quad \det(DH_z) \geq C \ell^{-N} e^{-Nh}, \forall z \in U \\
 \label{PartII2term 3} \quad
\norm{(DH_z)^{tr}(v)} \leq C \ell^{-N}e^{N h} \norm{v}, \forall z \in U, v \neq 0, (DH_z)^{tr}(v) \notin \hat{\CC}_{-}
\eary

We have the following.
\begin{prop} \label{PartII2prop ly ineq chart} 
Given integers $r \geq 13, 3 \leq q+3 \leq p < \frac{r}{2} - 3$. 
For any $h \in (0, \log \ell)$, integer $N_0 > 0$, any $f \in \cU^{h,N_0}$, any $u \in C^{r-1}_0(R)$, for any $a,b \in A$, we have
\aryst
\norm{P_{a,b}u}_{\Theta,p,q} \leq C\norm{u}_{\Theta,p,q}
\earyst
If in addition that $q \geq 1$, then for any 
\ary\label{PartII2essential spectrum radius}
 \tilde{m}  > m_0(f,h,p,q) := \max(e^{(2q+3)h} m(f), e^{(2p+3)h}\ell^{-(2p-1)}) 
\eary
we have
\aryst
\norm{P_{a,b}u}_{\Theta,p,q} \leq C \tilde{m}^{n/2} \norm{u}_{\Theta,p,q} + C_n \norm{u}_{\Theta', p-1,q-1}
\earyst
\end{prop}
\begin{proof}
The proof is an easy adaptation of Lemma 2.6 in \cite{Ts} using Lemma \ref{PartII2lemma ly-ineq local} instead of Lemma 2.4 in \cite{Ts}. We will only give a sketched proof. The reader is referred to \cite{Ts} for details.

By \eqref{PartII2essential spectrum radius}, we have $e^{-(2q+3)nh}\tilde{m}^{n} > m(f)^{n}$.
By \eqref{PartII2m is less that its finite scale approx}, we can assume that $n$ is sufficiently large, so that $m(f,n) < e^{-(2q+3)nh}\tilde{m}^{n}$.
We will choose a covering of the closure of $R$ by finitely many little open cubes in $Q$ with intersection multiplicity bounded by $10$, denoted by $\{D(\omega)\}_{ \omega \in \cA}$. Take a family of $C^{\infty}$ functions $\{ g_{\omega} : \R^2 \to [0,1] \}_{\omega \in \cA}$ such that $supp(g_{\omega}) \Subset D(\omega)$ and $\sum_{\omega \in \cA} g_{\omega}(z) = 1$ for any $z \in R$.

Fix $u \in C^{r-1}(R)$, $a,b \in A$, $\omega \in \cA$, we denote the connected components of the preimage  $f^{-n}(\kappa_aD(\omega)) \bigcap R_b$ by $\kappa_b(D(\omega,i)), 1 \leq i \leq I(\omega)$, where $D(\omega, i) \subset R$ are open sets. By letting $D(\omega)$ to be small, we can ensure that for each $1 \leq i \leq I(\omega)$,  $\kappa_a^{-1} f^n \kappa_b : D(\omega,i) \to  D(\omega)$ is a $C^{r}$ injection; and  
 by setting $i \pitchfork_{\omega}j$ if 
\ary \label{PartII2chooseinvarsebranches} \quad
\overline{Df_{w}^{n}(\CC_0)} \bigcap \overline{Df_{w'}^{n}(\CC_0)} = \{0\}, \quad \forall w \in D(\omega,i),w' \in D(\omega,j),
\eary
 for each $i$ there are at most $m(f,n)$ many $j$ such that $i \not\pitchfork_{\omega}j$.

 We define functions $\{ g_{\omega, i} : \R^2 \to [0,1] \}_{i=1}^{I(\omega)}$ by 
\aryst
g_{\omega, i}(y) = \begin{cases}( g_{\omega} \rho_{a} \circ \kappa_a)(\kappa_{a}^{-1}f^{n}\kappa_b(y)) , \quad y \in D(\omega, i) \\ 0, \quad\mbox{ otherwise} \end{cases}
\earyst
We claim that $g_{\omega,i}$ is $C^{r-1}$ in $R$.
Indeed, it is clear that $g_{\omega, i}$ is $C^{r-1}$ in  $D(\omega, i)$ and continuously extends up to the boundary. Moreover, for any $z \in \partial D(\omega, i) \bigcap R$, we have $\kappa_a^{-1} f^n \kappa_b (z) \in \partial D(\omega)$, for otherwise an open neighborhood of $z$ would be mapped into $D(\omega)$, thus $z \in D(\omega, i)$, a contradiction. While $g_{\omega}$ vanishes on an open neighbourhood of $\partial D(\omega)$. This implies our claim, and also proves that $g_{\omega, i}$ vanish in an open neighbourhood of  
$\overline{\partial D(\omega,i) \bigcap R}$.
Define $
g_{\omega} := \sum_{i=1}^{I(\omega)}g_{\omega,i}$.
Since $D(\omega,i), 1 \leq i \leq I(\omega)$ are mutually disjoint, we have $0 \leq g_{\omega} \leq 1$.
Define
\aryst
g(y) = \sum_{\omega \in \cA} g_{\omega}(y) = \begin{cases}  \rho_a f^{n}\kappa_b(y) , \quad y \in \kappa_b^{-1}(f^{-n}(R_a) \bigcap R_b), \\ 0, \quad\mbox{otherwise} \end{cases}
\earyst
We can easily verify that $0 \leq g \leq 1$ and $g$ is $C^{r-1}$ in the interior of $R$. 

Let $u \in C^{r-1}_0(R)$. We have the following,
\enmt
\item for any $1 \leq i \leq I(\omega)$, define $u_{\omega, i} := g_{\omega, i} u $. Then we have 
$
supp(u_{\omega, i}) \Subset D(\omega,i),
$ since $g_{\omega,i}$ vanish in an open neighbourhood of $\overline{\partial D(\omega,i) \bigcap R}$ and $u$ vanish in an open neighbourhood of $\partial R$.
\item for any $1 \leq i \leq I(\omega)$, define
\aryst
v_{\omega,i}(x) = \begin{cases} (u_{\omega, i}  \det(Df^{n}\circ \kappa_b(\cdot))^{-1})( \kappa_b^{-1}f^{-n}\kappa_a(x)), x \in \kappa_a^{-1}f^{n}\kappa_b(D(\omega, i)) \\ 0, \mbox{ otherwise}\end{cases}
\earyst
We have that $v_{\omega, i} \in C^{r-1}(R)$,
\item let $v_{\omega} := \sum_{i=1}^{I(\omega)} v_{\omega, i} = g_{\omega}P^n_{a,b}u$,
\item we have $P^n_{a,b}u = \sum_{\omega \in \cA} v_{\omega}$.
\eenmt

Denote by $S = \kappa_b^{-1}f^{-n}\kappa_a : \kappa_a^{-1}f^{n}\kappa_b(D(\omega, i)) \to D(\omega)$.
By $f \in \cal U^{h, N_0}$ and \eqref{PartII2term 1},  \eqref{PartII2term 2}, \eqref{PartII2term 3}, we have for any $n \geq 1$
\aryst
&&\norm{(\det(Df^{n}))^{-1}}_{L^{\infty}} \leq C \ell^{-n} e^{nh}, \\
&&\gamma(S) \geq C \ell^{-n} e^{-nh}, \quad \norm{DS} \leq C e^{nh}, \quad \Lambda(S, \hat{\Theta}) < C \ell^{-n}e^{nh}
\earyst
Then by Lemma \ref{PartII2lemma ly-ineq local} and our hypothesis that $p,q \in [0, \frac{r}{2} - 3)$, we have
\aryst
\norm{v_{\omega,i}}_{\hat{\Theta}, p, q} \leq C_n\norm{u_{w,i}}_{\check{\Theta}, p, q} 
\earyst
Moreover, if $q \geq 1$, then 
\enmt
\item $\norm{v_{\omega,i}}_{\hat{\Theta}, q}^{-} \leq C {\ell}^{-\frac{n}{2}}e^{(q+\frac{3}{2})nh} \norm{u_{w,i}}_{\check{\Theta}, p, q} + C_n\norm{u_{\omega, i}}_{\check{\Theta}, p-1, q-1}$,
\item $\norm{v_{w,i}}_{\hat{\Theta}, p}^{+} \leq C {\ell}^{-(p + \frac{1}{2})n} e^{(p + \frac{3}{2})nh}  \norm{u_{w,i}}_{\check{\Theta}, p, q} + C_n\norm{u_{\omega, i}}_{\check{\Theta}, p-1, q-1}$,
\eenmt
We choose polarisations $\{\Theta(\omega,i) = (\CC_{\omega,i,\pm}, \varphi_{\omega,i,\pm})\}_{i=1}^{I(\omega)}$ such that for all $1 \leq i \leq I(\omega)$, 
\aryst
((Df^n_x)^{tr})^{-1}(\R^2 \setminus \check{\CC}_{+}) \Subset \CC_{\omega,i, -} \Subset (\R^2 \setminus \CC_{\omega,i,+}) \Subset \hat{\CC}_{-}, \forall x \in D(\omega,i)
\earyst
and 
\ary \label{PartII2conedisjoint}
\overline{(\R^2 \setminus \CC_{\omega,i, +})} \bigcap \overline{(\R^2 \setminus \CC_{\omega,j, +})} = \{0\}, \mbox{ if } i \pitchfork_{\omega}j
\eary
It is possible by $i \pitchfork_{\omega} j$, \eqref{PartII2check cc cc hat cc relation 2}, \eqref{PartII2dfz-1trdfzsub} and \eqref{PartII2chooseinvarsebranches}.

It is clear that $\hat{\Theta} < \Theta(\omega,i)$ for all $\omega \in \cA, 1 \leq i \leq I(\omega)$.
\begin{lemma}\label{PartII2lem transversality and norm}
If $i \pitchfork_{\omega} j$, we have
\aryst
\sum_{n \geq 0} 2^{2nq}|(\psi_{\hat{\Theta}, n, -}(D)v_{\omega, i},  \psi_{\hat{\Theta}, n, -}(D)v_{\omega, j} )_{L^2}| \leq C \norm{v_{\omega, i}}_{\Theta(\omega, i), p-1, q-1}  \norm{v_{\omega, j}}_{\Theta(\omega, j), p-1, q-1}
\earyst
\end{lemma}
\begin{proof}
The proof is similar to that of Lemma 2.7 in \cite{Ts}.
For $k = i,j$, put $w_{k,n} = \psi_{\hat{\Theta}, n,-}(D)v_{\omega,k}, w'_{k,n} = \psi_{\Theta(\omega,k), n,-}(D)v_{\omega,k}, w''_{k,n} = w_{k,n} - w_{k,n}'$. By \eqref{PartII2conedisjoint}, for $n > 0$ we have $(w'_{i,n}, w'_{j,n})_{L^2} = 0$. While $2^{(q-1)n}\norm{w'_{i,n}}_{L^2} \leq \norm{v_{\omega,i}}_{\Theta(\omega,i), p-1,q-1}$ and $2^{(p-1)n}\norm{w''_{i,n}}_{L^2} \leq \norm{v_{\omega,i}}_{\Theta(\omega,i), p-1,q-1}$. We have the similar thing for $j$. Thus by $p \geq q+3$,
\aryst
|(w_{i,n}, w_{j,n})_{L^2}| &\leq& (2 \cdot 2^{-(q-1 + p-1)n} + 2^{-(2p-2)n})\norm{v_{\omega, i}}_{\Theta(\omega, i), p-1, q-1}  \norm{v_{\omega, j}}_{\Theta(\omega, j), p-1, q-1} \\
&\leq& 10\cdot 2^{-(2q+1)n}\norm{v_{\omega, i}}_{\Theta(\omega, i), p-1, q-1}  \norm{v_{\omega, j}}_{\Theta(\omega, j), p-1, q-1}
\earyst
The lemma follows from direct computations.
\end{proof}
By $((Df^n_x)^{tr})^{-1}(\R^2 \setminus \check{\CC}_{+}) \Subset \CC_{\omega,i, -} $ and Lemma \ref{PartII2lemma ly-ineq local}, we have for $q \geq 1$ and all $1 \leq i \leq I(\omega)$,
\ary\label{PartII2vomegaiuomegaip-1q-1}
 \norm{v_{\omega,i}}_{\Theta(\omega,i), p-1,q-1} \leq C_n\norm{u_{\omega,i}}_{\check{\Theta},p-1,q-1}
\eary

Then the rest of the proof follows almost exactly that of Lemma 2.6 in \cite{Ts}.
By Lemma \ref{PartII2lem norm unit partition}, for any $u \in C_{0}^{r-1}(R)$, we have for any $n \geq 1$,
\ary
\label{PartII2pnabuthetapq2} \\
\norm{P^n_{a,b}u}_{\Theta, p, q}^2 &=& \norm{\sum_{\omega \in \cA} v_{\omega}}_{\Theta, p, q}^2 \leq C\sum_{\omega \in \cA} \norm{ v_{\omega}}_{\hat{\Theta}, p, q}^2 + C_n\sum_{\omega \in \cA} \norm{ v_{\omega}}_{\hat{\Theta}, p-1, q-1}^2 \nonumber
\eary
By Lemma \ref{PartII2lem transversality and norm}, Cauchy's inequality and \eqref{PartII2vomegaiuomegaip-1q-1},
\ary
(\norm{ v_{\omega}}_{\hat{\Theta}, q}^{-})^2 
&\leq& C \sum_{i} \sum_{j \not\pitchfork_{\omega} i} \ell^{-n} e^{(2q+3)nh} \frac{\norm{u_{\omega, j}}_{\check{\Theta}, p, q}^2 + \norm{u_{\omega,i}}_{\check{\Theta}, p, q}^2}{2} + C_n \sum_{i} \norm{u_{\omega, i}}_{\check{\Theta}, p-1, q-1}^2  \nonumber  \\
&&\leq C \tilde{m}^{n}  \sum_{i} \norm{u_{\omega, i}}_{\check{\Theta}, p, q}^2  + C_n \sum_{i} \norm{u_{\omega, i}}_{\check{\Theta}, p-1, q-1}^2 \label{PartII2vomegahatthetaq-2}
\eary
By Lemma \ref{PartII2lem norm unit partition} and the calculation of $g$, 
\ary
\sum_{\omega \in \cA} \sum_{i} \norm{u_{\omega, i}}_{\check{\Theta}, p, q}^2 &\leq& C\norm{u}_{\Theta, p, q}^2 + C_n\norm{u}_{\Theta', p-1, q-1}^2 \label{PartII2sumomegaincAsuminormuomegai}\\
\sum_{\omega \in \cA} \sum_{i} \norm{u_{\omega, i}}_{\check{\Theta}, p-1, q-1}^2 &\leq& C_n \norm{u}_{\Theta', p-1,q-1}^2 \label{PartII2sumomegaincAchecktheta}
\eary
By \eqref{PartII2vomegahatthetaq-2}, \eqref{PartII2sumomegaincAsuminormuomegai}, \eqref{PartII2sumomegaincAchecktheta}, we obtain
\ary \label{PartII2sumomegahatq-111}
\sum_{\omega \in \cA}(\norm{ v_{\omega}}_{\hat{\Theta}, q}^{-})^2 \leq C \tilde{m}^2\norm{u}_{\Theta, p,q}^2 + C_n\norm{u}_{\Theta',p-1,q-1}^2
\eary
We have
\aryst
&&(\norm{ v_{\omega}}_{\hat{\Theta}, p}^{+})^2 \leq \ell^{2n} \sum_{i} (\norm{v_{\omega,i}}^{+}_{\hat{\Theta}, p})^2 \\
&\leq& C\sum_{i}\ell^{-n(2p-1)}e^{(2p+3)nh} \norm{u_{\omega,i}}_{\check{\Theta}, p, q}^2 +  C_n \sum_{i} \norm{u_{\omega, i}}_{\check{\Theta}, p-1, q-1}^2 
\earyst
Again by  \eqref{PartII2sumomegaincAsuminormuomegai}, \eqref{PartII2sumomegaincAchecktheta}, we obtain
\ary \label{PartII2sumomegahatp+222}
\sum_{\omega \in \cA}(\norm{ v_{\omega}}_{\hat{\Theta}, p}^{+})^2&\leq&C\tilde{m}^{n} \norm{u}_{\Theta, p, q}^2 +  C_n \norm{u}^2_{\Theta', p-1, q-1} 
\eary

Finally by Lemma \ref{PartII2lem norm unit partition}, Lemma \ref{PartII2lemma ly-ineq local} and $\hat{\Theta} < \Theta(\omega,i)$,  \eqref{PartII2vomegaiuomegaip-1q-1}, we have
\aryst
\norm{ v_{\omega}}_{\hat{\Theta}, p-1, q-1}^2 \leq C_n\sum_{i} \norm{v_{\omega,i}}_{\Theta(\omega,i), p-1,q-1}^2 \leq C_n\sum_{i}\norm{u_{\omega,i}}_{\check{\Theta},p-1,q-1}^2
\earyst
By \eqref{PartII2sumomegaincAchecktheta}, 
\ary \label{PartII2sumomegahatthetap'q'}
\sum_{\omega \in \cA}\norm{ v_{\omega}}_{\hat{\Theta}, p-1, q-1}^2 \leq C_n \norm{u}_{\Theta',p-1,q-1}^2
\eary
Then lemma follows from \eqref{PartII2pnabuthetapq2}, \eqref{PartII2sumomegahatq-111}, \eqref{PartII2sumomegahatp+222}, \eqref{PartII2sumomegahatthetap'q'}.

\end{proof}
\begin{cor} \label{PartII2lasotayorkefinalform}
Given integers $r \geq 13, 3 \leq q+3 \leq p < \frac{r}{2} - 3$. Let $h, N_0, f$ be given by Proposition \ref{PartII2prop ly ineq chart}.
For any $u \in C^{\infty}(\T^2)$, we have,
\aryst
\norm{\PF u}_{\Theta, p, q} \leq C\norm{u}_{\Theta,p,q}
\earyst
If in addition that $q \geq 1$, then for any $\tilde{m}$ in \eqref{PartII2essential spectrum radius}, there exists $M > 0$ such that for any $u \in C^{\infty}(\T^2)$, any $n \in \N$,
\aryst
\norm{\PF^n u}_{\Theta, p, q} \leq C\tilde{m}^{\frac{n}{2}} \norm{u}_{\Theta, p, q} +CM^n\norm{u}_{\Theta', p-1, q-1}
\earyst
\end{cor}
\begin{proof}
We choose an arbitrary $\overline{m} \in (m_0(f,h,p,q), \tilde{m})$ ( recall \eqref{PartII2essential spectrum radius}).
By \eqref{PartII2normpfuthetapq2} and Proposition \ref{PartII2prop ly ineq chart} we have
\ary \label{PartII2mPF u2Theta,p,qleq}
\norm{\PF u}^2_{\Theta,p,q} \leq C  \sum_{b \in A} \norm{(\rho_b u) \circ \kappa_b}_{\Theta, p, q}^2  
\leq C \norm{ u}_{\Theta, p, q}^2 
\eary
and for $q \geq 1$,
\aryst
\norm{\PF^N u}^2_{\Theta,p,q} &\leq& C \overline{m}^{N} \sum_{b \in A} \norm{(\rho_b u) \circ \kappa_b}_{\Theta, p, q}^2  + C_N \sum_{b \in A}\norm{(\rho_b u) \circ \kappa_b}_{\Theta', p-1, q-1}^2 \\
&\leq& C \overline{m}^{N} \norm{ u}_{\Theta, p, q}^2 + C_N \norm{u}_{\Theta',p-1,q-1}^2
\earyst
Then
\ary \label{PartII2normpfnu}
\norm{\PF^N u}_{\Theta,p,q} \leq C \overline{m}^{\frac{N}{2} }\norm{ u}_{\Theta, p, q} + C_N \norm{u}_{\Theta',p-1,q-1}
\eary
We fix $N$ to be a large integer so that the coefficient of $\norm{u}_{\Theta,p,q}$ in \eqref{PartII2normpfnu} is less than $\tilde{m}^{\frac{N}{2}}$. Let $M$ be a large constant to be chosen later. We will inductively prove that for all integer $l \geq 1$,
\ary \label{PartII2normpfnlu}
\norm{\PF^{Nl} u}_{\Theta,p,q} \leq \tilde{m}^{\frac{Nl}{2}} \norm{ u}_{\Theta, p, q} + M^{Nl} \norm{u}_{\Theta',p-1,q-1}
\eary
This is true for $l=1$ by \eqref{PartII2normpfnu} and by letting $M > C_N^{\frac{1}{N}}$. Assume that \eqref{PartII2normpfnlu} is prove for $l$. Then by \eqref{PartII2mPF u2Theta,p,qleq} and \eqref{PartII2normpfnlu} we have
\aryst
\norm{\PF^{N(l+1)} u}_{\Theta,p,q} &\leq& \tilde{m}^{\frac{N}{2}} \norm{ \PF^{Nl}u}_{\Theta, p, q} + M^{N} \norm{\PF^{Nl}u}_{\Theta',p-1,q-1} \\
&\leq& \tilde{m}^{\frac{N}{2}}(\tilde{m}^{\frac{Nl}{2}} \norm{u}_{\Theta,p,q} + M^{Nl}\norm{u}_{\Theta',p-1,q-1}) + M^N C^{Nl}\norm{u}_{\Theta', p-1, q-1} \\
&\leq& \tilde{m}^{\frac{N(l+1)}{2}} \norm{u}_{\Theta,p,q} + M^{N(l+1)}\norm{u}_{\Theta',p-1,q-1}
\earyst
The last inequality follows by letting $M > 10\max(\tilde{m}^{\frac{1}{2}}, \tilde{m}^{-\frac{1}{2}}C, C)$. This completes the induction. Our corollary then follows by letting $C$ in the second inequality of our corollary  to be large depending on $N$.
\end{proof}

\subsection{Gou\"ezel-Liverani's perturbation lemma}
We recall an abstract result from \cite[Section 8]{GoLi}.
Let $\cB^{0} \supset \cB^1 \supset \cdots \supset \cB^s, s \in \N$, be a finite family of Banach spaces, let $\{\cal L_{t}\}_{t \in (-1,1)}$ be a family of operators acting on the above Banach spaces. Moreover, assume that
\enmt
\item there exist $M > 0$, for all $t \in (-1,1)$, $\norm{\cal L^{n}_t u}_{\cB^0} \leq C_0M^n \norm{u}_{\cB^0}$,
\item there exists $\alpha \in (0, M)$, for all $t \in (-1,1)$, $\norm{\cal L^{n}_t u}_{\cB^1} \leq C_0\alpha^{n} \norm{u}_{\cB^1}+ C_0M^{n}\norm{f}_{\cB^0}$,
\item there exist operators $Q_1,\cdots, Q_{s-1}$ satisfying 
\aryst
\norm{Q_j}_{\cB^{i} \to \cB^{i-j}} \leq C_1, \forall j=1,\cdots, s-1, i=j,\cdots, s 
\earyst
\item moreover, define $\Delta_0(t) := \cal L_t$ and $\Delta_j(t) := \cal L_t - \cal L_0 - \sum_{k=1}^{j-1} t^{k} Q_k$ for $j\geq 1$, we have
\aryst
\norm{\Delta_{j}(t)}_{\cB^{i} \to \cB^{i-j}} \leq C_1|t|^{j}, \forall t \in (-1,1), j=0,\cdots, s, i=j,\cdots, s
\earyst

\eenmt
In this case, we say that  $\{\cal L_t\}_{t \in (-1,1)}$ is \textit{$(\alpha, M, C_0, C_1)$ adapted to} $\{\cB^i\}_{0 \leq i \leq s}$.

For any integer $1 \leq k \leq s$, any $t \in (-1,1)$, any $\varrho > \alpha$ and $\delta > 0$, denote
\aryst
V_{\delta, \varrho} = \{z \in \mathbb{C}| |z| \geq \varrho, d(z, Sp(\cL_0 : \cB^k \to \cB^k)) \geq \delta, \forall k=1,\cdots, s\}
\earyst

The following theorem in proved in \cite{GoLi}.
\begin{thm}[ Theorem 8.1 in \cite{GoLi} ] \label{PartII2thm goli perturbations}
Given a family of operators $\{\cal L_t\}_{t \in (-1,1)}$ that is $(\alpha, M, C_0, C_1)$ adapted to $\{\cB^{i}\}_{0 \leq i \leq s}$ and set
\aryst
R_{s}(t) = \sum_{k=0}^{s-1} t^{k} \sum_{l_1 + \cdots + l_j = k} (z-\cal L_{0})^{-1} Q_{l_1} (z-\cal L_0)^{-1} \cdots (z-\cal L_0)^{-1} Q_{l_j}(z-\cal L_0)^{-1}
\earyst
then for all $\varrho > \alpha, \delta > 0$, there exists $\eta > 0, C_2 = C_2(\alpha, M, C_0, C_1, \varrho, \delta) > 0, t_0 = t_0(\alpha, M, C_0, C_1, \varrho, \delta) > 0$ such that
for all $z \in V_{\delta,\varrho}$ and $t \in (-t_0, t_0)$, we have that
\aryst
\norm{(z-\cal L_t)^{-1} - R_s(t)}_{\cal B^s \to \cal B^{0}} \leq C_2 |t|^{s-1+\eta}
\earyst
\end{thm}

Let $r \geq r'+2 \geq 3$. Given any $C^{r'+1}$ family in $C^r(\T^2)$, denoted by $\{f_t\}_{t \in (-1,1)}$.  For any $t \in (-1,1)$, we denote $\PF_t = \PF_{f_t}$.

By Taylor's formula, for each $u \in C^{r}(\T^2)$, for all $1 \leq k \leq r'+1$, we have
\ary \label{PartII2taylor 1}
\PF_t u = \sum_{j=0}^{k-1} \frac{1}{j!} \partial_t^{j}\PF_tu|_{t=0} + \int_{0}^t dt_1 \cdots \int_{0}^{t_{k-1}} dt_{k} (\partial_t^{k} \PF_t u)(t_{k}) 
\eary

For any $k \geq 1$, any $\alpha = (\alpha_1, \cdots, \alpha_k) \in \{1,2\}^k$, we denote $|\alpha| = k$ and define by $\partial^{\alpha}$ the linear operator from $C^{\infty}(\T^2)$ to $C^{\infty}(\T^2)$ that
\ary \label{PartII2defpartial}
\partial^{\alpha} u = \partial_{\alpha_k} \cdots \partial_{\alpha_1}u
\eary

Then there exist for each $1 \leq k \leq r'+1$, functions $J_0(k,t,x)$ which are $C^{0}$ in $t$ and $C^r$ in $x$, and for each multi-index $\alpha, 1 \leq |\alpha| \leq k$, functions  $J_{\alpha}(k,t,x)$ which are $C^{|\alpha|}$ in $t$ and $C^{r-1}$ in $x$, such that for all $ t_0 \in (-1,1)$
\ary 
\nonumber 
\partial_{t}^{k}\PF_{t} u (x)|_{t=t_0}&=& J_0(k, t_0, x)(\PF_{t_0} u)(x)+\sum_{j=1}^{k} \sum_{|\alpha| = j} J_{\alpha}(k, t_0, x)(\PF_t \partial^{\alpha}u)(x) |_{t=t_0} \\
& =&: k! Q_{k, t_0}u(x), \label{PartII2taylor 2} 
\eary
Moreover, for $1 \leq k \leq r'+1, 1 \leq j \leq k, \alpha \in \{1, 2\}^j$, we have
\ary \label{PartII2bound for j}
\sup_{t \in (-1,1)}\norm{J_{\alpha}(k,t,\cdot)}_{C^{r-1}(\T^2)} \leq C(\norm{\{f_t\}_{t \in (-1,1)}}_{r'+1,r})
\eary
and
\ary \label{PartII2bound for j0}
\sup_{t \in (-1,1)}\norm{J_{0}(k,t,\cdot)}_{C^{r-1}(\T^2)} \leq C(\norm{\{f_t\}_{t \in (-1,1)}}_{r'+1,r})
\eary

We need the following lemma.
\begin{lemma} \label{PartII2lem partial is bounded}
Let $\Theta, \Theta'$ be two polarisations such that $\Theta < \Theta'$, and let $p,q \in \N$, $q \leq p$.
For any $k \geq 0$, for any multi-index $\alpha \in \{1,2\}^k$ ( when $k=0$, we set $\alpha = \emptyset$ and $\partial^{\alpha} = 1$ ),  for any $J \in C^{p+k}(\T^2)$, $J\partial^{\alpha}$ is a bounded operator from $W_{\Theta', p+ k, q+ k}(\T^2)$ to $W_{\Theta, p, q}(\T^2)$ with norm bounded by $C = C(\Theta, \Theta', p, q, k, \norm{J}_{C^{p+k}})$.
\end{lemma}
\begin{proof}
In the following we will consider $\Theta, \Theta', p, q$ to be fixed, so that we will not express the dependence of varies constants on them.

We prove our lemma by induction on $k$. We denote $\alpha = (\alpha_1,\cdots, \alpha_k)$.
For any $u \in C^{\infty}(\T^2)$, $a \in A$, we have
\aryst
\partial^{\alpha}((\rho_a J u) \circ \kappa_a) = \rho_a \circ \kappa_a (J \partial^{\alpha}u )\circ \kappa_a +( \tilde{\rho} u)\circ \kappa_a +  \sum_{\beta, 1 \leq |\beta| \leq k-1} (\tilde{\rho}_{\beta} \partial^{\beta}u) \circ \kappa_a
\earyst
where $\tilde{\rho} \in C^{p}_0(R_a), \tilde{\rho}_{\beta} \in C^{p + |\beta|}_0(R_a), \forall \beta, 1 \leq |\beta| \leq k-1$. Moreover, it is direct to see that for all $\beta, 1 \leq |\beta| \leq k-1$,
\ary\label{PartII2tilderhotilderhobetanorm}
\norm{\tilde{\rho}}_{C^p} \lesssim \norm{\rho_a}_{C^{p+k}}\norm{J}_{C^{p+k}}, \quad \norm{\tilde{\rho}_{\beta}}_{C^{p+|\beta|}} \lesssim \norm{\rho_a}_{C^{p+k}}\norm{J}_{C^{p+k}}, 
\eary

For any $v \in C^{\infty}_0(\R^2)$, any $\zeta = (\zeta_1, \zeta_2) \in \R^2$, we have 
\aryst
\cal F(\partial^{\alpha}v)(\zeta) = (2\pi i)^{k}( \prod_{j=1}^{k} \zeta_{\alpha_j} ) \cal F(v)(\zeta)
\earyst
Then for any $(n,\sigma) \in \N \times \{+,-\}$, for any $\zeta \in supp(\psi_{\Theta, n, \sigma})$, we have
\ary \label{PartII2cal f partial alpha}
|\cal F(\partial^{\alpha}((\rho_a J u)\circ \kappa_a))(\zeta)| \leq C2^{(n+1)k}\cal F((\rho_a J u)\circ \kappa_a)(\zeta)
\eary
By Lemma \ref{PartII2lem norm unit partition}, there exists $C'_1 = C'_1(k, \norm{\tilde{\rho}}_{C^p})$ such that 
\ary \label{PartII2norm tilde rho beta u}
&&\norm{(\tilde{\rho}  u )\circ \kappa_a}_{\Theta, p, q}\leq C'_1 \norm{u}_{\Theta',p,q} 
\eary
We have
\enmt
\item[$(i)$] If $k=0$, then the boundedness of $J$ from $W_{\Theta', p, q}(\T^2)$ to $W_{\Theta, p, q}(\T^2)$ follows from Lemma \ref{PartII2lem norm unit partition}.
\item[$(ii)$] If $k=1$, then the boundedness of $J\partial^{\alpha}$ from $W_{\Theta', p+1, q+1}(\T^2)$ to $W_{\Theta, p, q}(\T^2)$ follows from $(i)$, \eqref{PartII2tilderhotilderhobetanorm},\eqref{PartII2cal f partial alpha}, \eqref{PartII2norm tilde rho beta u} and Parseval's identity.
\eenmt

Otherwise, assume that our lemma is true for $1,\cdots, k-1$. By \eqref{PartII2norm tilde rho beta u} for $(\tilde{\rho}_{\beta}, \partial^{\beta}u)$ in place of $(\tilde{\rho}, u)$ and the induction hypothesis, there exist a constant $C'_2$ depending only on $k-1$ and $\sup_{\beta, 1 \leq |\beta| \leq k-1}\norm{\tilde{\rho}_{\beta}}_{C^{p+|\beta|}}$ such that 
\ary \label{PartII2norm tilde rho beta partial}
\norm{(\tilde{\rho}_{\beta}\partial^{\beta}u) \circ \kappa_a}_{\Theta, p, q} &\leq& C'_2 \norm{u}_{\Theta', p+k, q+k} \nonumber
\eary
Then we verify our lemma for $k$ by \eqref{PartII2tilderhotilderhobetanorm}, \eqref{PartII2cal f partial alpha}, \eqref{PartII2norm tilde rho beta u},  \eqref{PartII2norm tilde rho beta partial} and Parseval's identity. This completes the induction and thus conclude the proof.
\end{proof}

\begin{lemma}\label{PartII2cor partial op is bounded}
Let $r,r', \ell, \gamma_0, \theta, f$ be given by Proposition \ref{PartII2prop diff}. Then there exist $C, M > 0, \alpha \in (0,1)$,  an open neighbourhood of $f$ in $C^{r}(\T^2, \T^2)$ denoted by  $\cU$, and Banach spaces $\cB^0 \supset \cB^1 \supset \cdots \supset \cB^{r'+1}$ satisfy that  for all $2 \leq i \leq r'+1$, $\cB^{i} \subset C^{i-2}(\T^2)$ and for all $1 \leq i \leq r'+1$, the inclusion $\cB^{i} \subset \cB^{i-1}$ is compact.
Moreover, for any $C^{r'+1}$ family in $\cU$, denoted by $\{f_t\}_{t \in (-1,1)}$, there exists constant $C_1 > 0$ depending only on $\norm{\{f_t\}_{t \in (-1,1)}}_{r'+1,r}$, such that $\{\PF_{f_t}\}_{t \in I}$ is  $(\alpha,  M, C, C_1)$  adapted to$\{\cB^i\}_{i=0}^{r'+1}$.
\end{lemma}
\begin{proof} 
Since $m(f) < 1$, by \eqref{PartII2m is less that its finite scale approx}, there exist $\overline{m} < 1$ and a $C^r$ open neighbourhood of $f$ in $C^r(\T^2,\T^2)$, denoted by $\cU_1$ such that $m(f') < \overline{m}$ for all $f' \in \cU_1$.

We choose $\check{\Theta} = (\check{\CC}_{\pm}, \check{\varphi}_{\pm}), \hat{\Theta} = (\hat{\CC}_{\pm}, \hat{\varphi}_{\pm})$ such that $\check{\CC} < \hat{\CC}$ and satisfy \eqref{PartII2check cc cc hat cc relation 2}, \eqref{PartII2check cc cc hat cc relation 3} for $f$ and $\CC_0 = \CC(\theta)$, and $\R(0,1)$ is contained in the interior of $\hat{\CC}_{-}$ except for at the origin. 
Then there exists an open neighbourhood of $f$ in $\cU_1$, denoted by $\cU_2$, such that  properties  \eqref{PartII2check cc cc hat cc relation 2}, \eqref{PartII2check cc cc hat cc relation 3} are satisfied for any $f' \in \cU_2$ in place of $f$.
Then fix a sequence of polarisations $\{\Theta_k = (\CC_{k,\pm}, \varphi_{k, \pm} ) \}_{k=0}^{r'+1}$ such that
\aryst
\check{\Theta} < \Theta_0 < \Theta_1 < \cdots < \Theta_{r'} < \Theta_{r'+1}  < \hat{\Theta}
\earyst
and define for $0 \leq k \leq r'+1$ that 
\ary \label{PartII2give cB k}
\cB^{k} = W_{\Theta_k, \lfloor \frac{r}{2} \rfloor -r'-5+k, k}
\eary
By Lemma \ref{PartII2lem norm basics}, we have $\cB^{i} \subset C^{i-2}(\T^2)$ for all $2 \leq i \leq r'+1$, and the inclusion $\cB^{i} \subset \cB^{i-1}$ is compact for all $1 \leq i \leq r'+1$.

For any integer $1 \leq k \leq r'+1$, we denote 
\aryst
\alpha_k = \max(e^{(2k+3)h}\overline{m}, e^{(r-2r'+2k-8)h} \ell^{-(r-2r'+2k-12)} )
\earyst
Then we have $\alpha_k > m_0(f', h, \lfloor \frac{r}{2} \rfloor -r'-5+k, k)$ for all $f' \in \cU_2$, where $m_0$ is defined in \eqref{PartII2essential spectrum radius}.

By $r' \leq \frac{r}{2}-9$, we have for any $0 \leq k \leq r'+1$, $(p,q) := (\lfloor \frac{r}{2} \rfloor -r'-5+k, k)$ that
\aryst
p \geq q+3, \quad p,q \in [0, \frac{r}{2}- 3)
\earyst
We take an arbitrary 
\aryst
h \in (0, \min( \frac{r-2r'-10}{r-2r'-6}\log \ell, \frac{-\log \overline{m}}{2r'+5}))
\earyst
It is direct to verify that 
\aryst
\alpha_0 :=\sup_{1 \leq k \leq r'+1} \alpha_k = \max(e^{(2r'+5)h}\overline{m}, e^{(r-2r'-6)h}\ell^{-(r-2r'-10)}) < 1
\earyst
We let $N_0 > 0$ be sufficiently large so that $\cU^{h,N_0}$ contains a $C^r$ open neighbourhood of $f$, denoted by $\cU_3$. We assume that $\cU_3 \subset \cU_2$.

 Take any $\alpha \in (\alpha_0^{\frac{1}{2}},1)$. We can apply Corollary \ref{PartII2lasotayorkefinalform} to see that there exist $C, M > 1$ such that for any $f' \in \cU_3$, any $u \in C^{\infty}(\T^2)$, any $n \geq 1$ that
 \aryst
 \norm{\PF_{f'}u}_{\cB^0} &\leq& M\norm{u}_{\cB^0}, \\
 \norm{\PF^n_{f'}u}_{\cB^k} &\leq& C\alpha^n\norm{u}_{\cB^k} +  CM^{n}\norm{u}_{\cB^{k-1}}, \quad 1 \leq k \leq r'+1
 \earyst

Given any  $C^{r'+1}$ family in $\cU_3$ denoted by $\{f_t\}_{t \in (-1,1)}$. Let $Q_{1,t},\cdots, Q_{r'+1,t}$ be defined by \eqref{PartII2taylor 2}. We then let 
\aryst
Q_j := Q_{j,0}, \quad \forall j=1,\cdots, r'
\earyst
We let $\Delta_0(t) = \PF_t$ for all $t \in (-1,1)$.
By \eqref{PartII2taylor 1}, for any $1 \leq j \leq r'+1$,
\ary \label{PartII2def delta for f t}
\Delta_{j}(t) = \int_{0}^t dt_1 \cdots \int_{0}^{t_{k-1}} dt_{k} (\partial_t^{k} \PF_t u)(t_{k})   \quad \forall t \in (-1,1)
\eary
Then by Lemma \ref{PartII2lem partial is bounded}, \eqref{PartII2def delta for f t}, \eqref{PartII2taylor 2}, \eqref{PartII2bound for j}, \eqref{PartII2bound for j0} there exists $C_1$ depending only on  $\norm{\{f_t\}_{t \in (-1,1)}}_{r'+1,r}$, such that
\aryst
\norm{Q_j}_{\cB^{i} \to \cB^{i-j}} &\leq& C_1, \forall j=1,\cdots, r', i=j,\cdots, r'+1  \\
\norm{\Delta_{j}(t)}_{\cB^{i} \to \cB^{i-j}} &\leq& C_1|t|^{j}, \forall t \in (-1,1), j=0,\cdots, r'+1, i=j,\cdots, r'+1
\earyst
This concludes the proof.
 
\end{proof}

Now we can prove Proposition \ref{PartII2prop diff}.
\begin{proof}[Proof of Proposition \ref{PartII2prop diff}:]
Let $f$ be given by Proposition \ref{PartII2prop diff}.
We let $\cB^0, \cdots, \cB^{r'+1}$ and $ \cU, C,M,\alpha$ be given by Lemma \ref{PartII2cor partial op is bounded}. 
Then for any $f \in \cU$, $\PF_f$ extends to bounded operator from $\cB^i$ to $\cB^i$ for all $0 \leq i \leq r'+1$. 

For any $f'$ in an open neighbourhood of $f$ in $C^{r}(\T^2, \T^2)$, we denote $s(f,f') = d_{C^r}(f,f')^{\frac{1}{r}}$, and  define $\{F^{f'}_{t}\}_{t \in (-1,1)}$, a $C^{r}$ family in $\cU$ by
\aryst
F^{f'}_{t}(x,y) = (1- \Pi(\frac{t}{s(f,f')}))f(x,y) + \Pi(\frac{t}{s(f,f')}) f'(x,y)
\earyst
where $\Pi \in C^{\infty}( \R , [0,1])$ such that 
$
\Pi(t) = \begin{cases} 0, \quad t < \frac{1}{3}, \\ 1, \quad t > \frac{2}{3}. \end{cases}
$, and we assume that $d_{C^r}(f,f') \ll 1$ so that the addition, the right hand side is interpreted as the linear interpolation between two nearby points $f(x,y), f'(x,y) \in \T^2$.
Then $\{F^{f'}_{t}\}_{t \in (-1,1)}$ constructed above satisfies that $F^{f'}_0 = f$,  $F^{f'}_{s(f,f')} = f'$, and$\norm{\{F^{f'}_t\}_{t \in (-1,1)}}_{r'+1,r} < C(f,\Pi)$. Then by Lemma \ref{PartII2cor partial op is bounded}, there exists a constant $C_1 > 0$ depending only on $f$ and $\Pi$ such that for any $f'$ sufficiently close to $f$ in $C^r$, $\{\PF_{F^{f'}_{t}}\}_{t \in (-1,1)}$ is $(\alpha, M, C, C_1)$ adapted to $\{\cB^{i}\}_{0 \leq i \leq r'+1}$.

By Lemma \ref{PartII2cor partial op is bounded} and Hennion's theorem in \cite{HeHe}, for $1 \leq i \leq r'+1$, $Sp(\PF_f : \cB^i \to \cB^i ) \bigcap \{z | |z| > \alpha \}$ contains isolated eigenvalues of finite multiplicity. By \eqref{PartII2eq pf operator}, $1-\PF_f$ is non-invertible in $\cB^i$ for all $1\leq i \leq r'+1$. Thus $1$ is an eigenvalue of $\PF_f$ with finite multiplicity in $\cB^i$. By our hypothesis that $f$ is ergodic with respect to the Lebesgue measure on $\T^2$, we have that for all $1 \leq i \leq r'+1$, $ Ker( (1 - \PF_f) : \cB^i \to \cB^i) = \R u$ for function $u \equiv 1 \in C^{\infty}(\T^2)$.

Let $\kappa > 0$ be a constant such that $1$ is the only eigenvalue of $\PF_f : \cB^i \to \cB^i$ in $\overline{B(1,\kappa)}$ for all $1 \leq i \leq r'+1$. By Theorem \ref{PartII2thm goli perturbations}, for all $1 \leq i \leq r'+1$, for all $z \in \partial B(1,\kappa)$, we have
\aryst
\sup_{f', d_{C^{r}}(f,f') \leq \epsilon}\norm{(z - \PF_{F^{f'}_{s(f,f')}})^{-1} - (z-\PF_f)^{-1} }_{\cB^{i} \to \cB^0} \to 0 \quad \mbox{ when } \epsilon\to 0
\earyst
Moreover, this convergence is uniform for all $z \in \partial B(1,\kappa)$. By Lemma \ref{PartII2cor partial op is bounded}, the inclusion $\cB^{i} \subset \cB^{i-1}$ is compact for all $1 \leq i \leq r'+1$. Then using the by-now standard argument in \cite{KeLi}, we see that  there exists $\delta > 0$ such that for all $f'$ such that $d_{C^r}(f,f') < \delta$, $\PF_{f'}$ has a unique simple eigenvalue in $\overline{B(1, \kappa)}$.

We now show that any $f'$ sufficiently close to $f$ has a unique SRB measure. 
We define for $f'$ sufficiently close to $f$ the spectral projection at $1$ by $\Pi_{f'}$. Then we have
\ary \label{PartII2spectral projection}
\Pi_{f'} = \frac{1}{2\pi i} \int_{|z-1|=\kappa} ( z - \PF_{f'})^{-1} dz
\eary
Moreover, denote $\rho_{f'} := \Pi_{f'}1 \in \cB^2$, then it is standard to see that $\rho_{f'} dLeb$ is $f'-$invariant. 

For all $f'$ sufficiently close to $f$ in $C^r$,  $\{\PF_{\cal F^{f'}_t}\}_{t \in (-1,1)}$ is $(\alpha, 2M, C, C_1)$ adapted to $\{\cB^{2}, \cB^3\}$.
Then by Theorem \ref{PartII2thm goli perturbations}, we have for all $z, |z-1|=\kappa$ that
\aryst
\norm{(z-\PF_{f'})^{-1} - (z-\PF_{f})^{-1}}_{\cB^3 \to \cB^2} \leq C_2 d_{C^{r}}(f,f')^{\frac{1}{r}}
\earyst
By Lemma \ref{PartII2lem norm basics} and \eqref{PartII2spectral projection}, we have
\aryst
\lim_{f', d_{C^r}(f' , f) \to 0 } \norm{\Pi_{f'}1 - \Pi_{f}1}_{C^0} \lesssim \lim_{f', d_{C^r}(f' , f) \to 0 } \norm{\Pi_{f'}1 - \Pi_{f}1}_{\cB^2} = 0
\earyst
While it is clear that $\Pi_{f}1 = 1$. This shows that for $f'$ sufficiently close to $f$ in $C^r(\T^2, \T^2)$, $\rho_{f'}(z) \geq \frac{1}{2}$ for all $z \in \T^2$. Then $\rho_{f'}dLeb$ is necessarily the unique SRB measure of $f'$. Let $\cU$ be a sufficiently small $C^r$ open neighbourhood of $f$ satisfying all the above conditions for $f'$.

Given a $C^{r'+1}$ family in $\cU$ denoted by $\{f_t\}_{t \in (-1,1)}$.  For any $\varphi \in C^{r}(\T^2)$, any $t \in (-1,1)$, we have
\aryst
\int \varphi d\mu_t = (\Pi_{f_t} 1, \varphi)_{L^2} = \frac{1}{2\pi i} \int_{|z-1| = \kappa} (( z - \PF_{f_t})^{-1}1, \varphi )_{L^2} dz
\earyst
Then our proposition follows from Lemma \ref{PartII2cor partial op is bounded} and Theorem \ref{PartII2thm goli perturbations}.
\end{proof}

\section{Nondifferentiability of u-Gibbs states}\label{PartII2Nondifferentiability of u-Gibbs states}

Our construction is inspired by a theorem of Halperin in the study of Anderson-Bernoulli model, stated in the Appendix of \cite{SiTa}. The argument in \cite{SiTa} is of spectral nature, and made essential use of the self-adjointness of the Schr\"odinger operators. Our argument is purely dynamical and focused on exploiting monotonicity and periodicity. This proof should shed some light on the study of the regularity of the density of states  of 1D Schr\"odinger operators with strongly mixing potentials.

\subsection{Markov partitions}

In this section, we define for $f$ that is either a partially hyperbolic system, or an Anosov system, or a strictly expanding map, a family of submanifolds that approximate the unstable manifolds of $f$. Note that for our later purpose, we only need to ensure that for any such submanifold, its image after long iterations can be almost decomposed into submanifolds in the same class. This makes our definition much simpler than the ones used in \cite{Do2}.

For any compact Riemannian manifold $X$, any precompact submanifold $\cD \subset X$, we denote by $Vol|_{\cD}$ the normalised volume form on $\cD$ induced by the restriction of the Riemannian metric on $\cD$. The normalisation ensures that $Vol|_{\cD}(\cD)=1$.

Let $f: X \to X$ be either a partially hyperbolic or an Anosov system. We denote by $E^{u}(x)$ the unstable subspace at $x$ of dimension $d_u$, and let $\cal K = \{\cal K(x) | E^{u}(x) \subset \cal K(x) \subset T_xX\}_{x \in X}$ denote a continuous family of cones containing $E^{u}(x)$ such that the closure of $Df(\cal K(x))$ is contained in the interior of $\cal K(f(x))$ except for the origin. 
\begin{defi}
Let $f, \cal K$ be given as above. For any $\varepsilon \in ( 0, 1), C_2 > 0$, we denote by $\bA_{\varepsilon, C_2, \cal K}(f)$ the set of submanifolds $\cD = \Phi((0, \varepsilon)^{d_u})$, where $\Phi : (0,2\varepsilon)^{d_u} \to X$ is a $C^2$ immersion such that $\norm{\Phi^{-1}}_{C^2},\norm{\Phi}_{C^2} < C_2, T\cD(x) \subset \cal K(x), \forall x \in \cD$.
\end{defi}

It is a standard fact that we can choose $\cal K$ such that for all $f'$ sufficiently close to $f$ in $C^1(X,X)$, the closure of $Df'(\cal K(x))$ is contained in the interior of $\cal K(f'(x))$ except for the origin. Moreover, there exists a constant $C_3$ depending only on $\norm{f}_{C^2}$ such that for any $\cD \in \bA_{\varepsilon, C_2, \cal K}(f)$, for any $n \geq 1$, let $\rho$ be the density of $(f^{n})_{*}(Vol|_{\cD})$ with respect to $Vol|_{f^{n}(\cD)}$. Then we have 
\aryst
|\log \rho(y_1) - \log \rho(y_2)| \leq C_3 d_{f^{n}(\cD)}(y_1,y_2), \quad \forall y_1,y_2 \in f^{n}(\cD)
\earyst
As a consequence, we have the following result. The proof is a standard exercise, which we omit.
\begin{lemma}\label{PartII2lem markov partition}
Let $C^2$ map $f : X \to X$ be either a partially hyperbolic system or an Anosov system. Then there exists a continuous family of cones $\cal K = \{ \cal K(x) | E^{u}(x) \subset \cal K(x) \subset T_x X\}_{x \in X}$, a constant $C_2 > 0$  such that the following is true. For any $x \in X$ the closure of $Df(\cal K(x))$ is contained in $\cal K(f(x))$ except for the origin . Moreover, for any $\kappa \in (0,1)$ there exists $\varepsilon_0 = \varepsilon_0(f,\kappa)$ with the following property.
 For any $\varepsilon \in (0,\varepsilon_0)$, there exist $ N_0 = N_0(\varepsilon) > 0$ and a $C^2$ open neighbourhood of $f$, denoted by $\cal U$, such that for any $f' \in \cal U$, any $\cD \in \bA_{\varepsilon, C_2, \cal K}(f)$, any integer $N > N_0$, there exist disjoint $\cD_1, \cdots, \cD_l \in \bA_{\varepsilon, C_2, \cal K}(f)$, constants $c_1, \cdots, c_l > 0$ such that for all $1 \leq i \leq l$, we have $\cD_i \subset f'^{N}(\cD)$, and
\aryst
\sum_{i=1}^{l} c_i Vol|_{\cD_i} \leq (f'^{N})_{*}(Vol|_{\cD}) \mbox{ and } \sum_{i=1}^{l} c_i > 1-\kappa
\earyst
\end{lemma}

In the following, for any $f$ that is either a partially hyperbolic system or an Anosov system, we will always choose $\cal K$, $C_2$ as in Lemma \ref{PartII2lem markov partition}. We will briefly denote $\bA_{\varepsilon, C_2, \cal K}(f)$ by $\bA_{\varepsilon}(f)$. When $f$ denotes a strictly expanding map, we define $\bA_{\varepsilon}(f)$ to be the collection of balls in $X$ of radius $\varepsilon$.

\subsection{Conditions for the construction}\label{PartII2Conditions for the construction}

As usual, we let $SL(2,\R)$ denote the special linear group acting on $\R^2$. We have a canonical action of $SL(2,\R)$ on $\mathbb P(\R^2)$. We use map $\psi : \mathbb P(\R^2) \to \T$, $\psi( \R(\cos \pi\theta, \sin \pi\theta)) = \theta$ to  identify $\mathbb P(\R^2)$ with $\T$. For any $H \in SL(2,\R)$, we denote $\widehat{H} = \psi H \psi^{-1} \in \diff^{\infty}(\T)$.

Let $H_{0} \in SL(2,\R)$ be a hyperbolic element with eigenvalues $e^{\alpha}, e^{-\alpha}$. Let $u_0, s_0 \in \T$ be respectively the sink and source of $\widehat{H_0}$. Then for all $H \in SL(2,\R)$ sufficiently close to $H_0$, $H$ is still a hyperbolic element. Let $u(H),s(H) \in \T$ be respectively the sink and source of $\widehat{H}$. Then we can easily verify that for all $H$ sufficiently close to $H_0$ the following is true,

\textbf{(HYP)} : \textit{ there exists a constant $c > 0$ such that  for any $\delta \in (0, \frac{1}{2})$,
 \aryst \widehat{H}^{n}(\T \setminus B(s(H), \delta)) \subset B(u(H), c\delta^{-1}e^{-n\alpha}), \quad \forall n \geq 1
\earyst}
 
We let $C_0  = \widehat{H_0} \in \diff^{\infty}(\T)$. Let $B_0 \in \diff^{\infty}(\T)$ satisfy
 that $B_0 u_0= s_0$. 
We denote by $\hat{C}, \hat{B} : \R \to \R$ respectively lifts of $C_0, B_0$. Let $\hat{u}_0,\hat{s}_0 \in [0,1)$ be respectively lifts of $u_0, s_0$. Without loss of generality, we can assume that : (1) $\hat{u}_0, \hat{s}_0$ are both fixed by $\hat{C}$, (2) $\hat{B}(\hat{u}_0) = \hat{s}_0$, (3) $\hat{s}_0 < \hat{u}_0$.

Let $M$ be a compact Riemannian manifold.  Let  map $g: M \to M$ be either a $C^r$ transitive Anosov diffeomorphism, or a $C^r$ strictly expanding  map. We denote by $m$ the unique SRB measure of $g$.

We denote by $p_1 : M \times \T \to M$, $p_2 : M \times \T \to \T$ be the canonical projections.
We let $f : M \times \T \to M \times \T$ be a $C^r$ map defined by
\ary \label{PartII2deff}
f(z,x) = (g(z), A(z,x))
\eary
where $A : M \times \T \to \T$ is a $C^r$ map.

We will assume that $f$ satisfies the following,

\enmt[(a)]
\item\label{PartII2laba} $\sup_{z \in M}\norm{DA(z,\cdot)}$ is small enough so that $f$ is partially hyperbolic,
\item\label{PartII2labb} there exists $\hat{f} : M \times \R \to M \times \R$ such that for any $(z,x) \in M\times \R$, $\hat{p}_2\hat{f}(z,x+1) = \hat{p}_2\hat{f}(z,x)+1$ and $\pi \hat{f} = f \pi$, where $\hat{p}_2 : M \times \R \to \R$, $\pi : M \times \R \to M \times \T$ are the canonical projections,
\item\label{PartII2labc} there is an open set $\cC\subset\{z | A(z,\cdot) = C_0 \} $, an open set $\cal B \subset \{z |  A(z, \cdot) = B_0 \}$ and constants $\kappa, c_0 \in (0,1)$, such that
\aryst
m(\cC) > 1 - \kappa, \quad m(\cB) > c_0, \quad m(\overline{\cC \bigcup \cB}) < 1
\earyst
\item \label{PartII2labd} there exist $z \in \cC \bigcap supp(m)$ and an integer $q \geq 1$ such that $g^{q}(z) = z$ and $f^{q}(z,s_0) \neq (z, s_0)$. 

We take an arbitrary $\varepsilon_1 \in (0,   d(z, \partial \cC) )$ and denote map $D: \T \to \T$ by
\aryst
D(x) = p_2f^{q}(z,x), \quad \forall x \in \T
\earyst
We let $\varepsilon_2 > 0$ be a constant such that $D(B(s_0,\varepsilon_2)) \bigcap B(s_0,\varepsilon_2) = \emptyset$.
Without loss of generality, we assume that $\varepsilon_1,\varepsilon_2 \in (0, \varepsilon_0(f, \frac{1}{2}))$, where $\varepsilon_0$ is given by Lemma \ref{PartII2lem markov partition}.

\item \label{PartII2labe} there exists a constant $\varepsilon \in (0, \min(\varepsilon_1, \varepsilon_2)/10)$ such that the following is true. For any $\cD \in \bA_{\varepsilon}(f)$, there exist disjoint $\cal E_1, \cdots, \cal E_l \in \bA_{\varepsilon}(f)$, and $d_1,\cdots, d_l > 0$ such that for all $1 \leq i \leq l$,
\aryst
\overline{\cal E_i} \subset f(\cD) \bigcap (\cC \times \T), \quad \sum_{i=1}^{l} d_i Vol|_{\cal E_i} \leq f_{*}Vol|_{\cD}, \quad \sum_{i=1}^l d_i > 1-\kappa
\earyst
Similarly, there exist disjoint $\cal F_1, \cdots, \cal F_k \in \bA_{\varepsilon}(f)$, and $h_1, \cdots, h_k > 0$,  such that for all $1\leq i\leq k$,
\aryst
\overline{\cal F_i} \subset f(\cD) \bigcap (\cB \times \T), \quad \sum_{i=1}^{k} h_i Vol|_{\cal F_i} \leq f_{*}Vol|_{\cD}, \quad \sum_{i=1}^k h_i > c_0
\earyst
\item \label{PartII2labf} Let $\varepsilon > 0$ be as in \eqref{PartII2labe}. There exists  closed interval $ J_1 \subset \T \setminus \{s_0\}$ such that there is a constant $K \in \N$,  such that for each $\cD \in \bA_{\varepsilon}(f)$, there exists $\cD'\in \bA_{\varepsilon}(f)$ satisfying $\overline{\cD'} \subset f^{K}(\cD) \bigcap (\cC \times J_1)$. We let $J$ be a closed interval contained in $\T \setminus \{s_0\}$ such that
$J_1 \Subset J$. 
We denote $\widehat{J} = \pi_{\R \to \T}^{-1}(J) \bigcap [0,1)$.

\item\label{PartII2labg}  for any $\varepsilon' > 0$, any sequence $\{\cD_i\}_{i \geq 0} \subset \bA_{\varepsilon'}(f)$, any strictly increasing $\{K_i\}_{i \geq 0} \subset \N$, the accumulating points of $(f^{K_n})_{*}(Vol|_{\cD_{n}})$ are contained in $uGibbs(f)$.
\eenmt

\begin{rema} \label{PartII2rema stability of conditions}
It is clear there exists $\sigma > 0$ such that properties \eqref{PartII2laba}, \eqref{PartII2labb}, \eqref{PartII2labc}, \eqref{PartII2labe}, \eqref{PartII2labf} are satisfied for any $f' : M \times \T \to M \times \T$ satisfying that $d_{C^r}(f,f')  < \sigma$, $p_1 f' = p_1 f$, and that $f'(z,\cdot) \neq f(z,\cdot)$ is contained in a $\sigma-$ball.
\end{rema}

\begin{rema} \label{PartII2rema condition a to e}
We now explain the applicability of the above conditions. 
Given any $C^r$ map $A : M \times \T \to \T$, we can make \eqref{PartII2laba} valid via replacing $g$ by any large power of $g$. 
Condition \eqref{PartII2labb} is valid for any $A$ that is $C^0$ close to  maps of the form $A_0 : M \times \T \to \T, A_0(z,x) = x+\varphi(z)$, and for any $g \in C^r(M, M)$. 
For any $\kappa, c_0 \in (0,1), \kappa + c_0 < 1$, we can choose $A$ satisfying condition \eqref{PartII2labc}  since $m$ has no atoms. 
The validity of \eqref{PartII2labd} is easily satisfied. For any $\varepsilon > 0$, we can make \eqref{PartII2labe} valid via replacing $g$ by any large power of $g$ : this is obvious  for strictly expanding $g$; for Anosov map $g$, this follows from \eqref{PartII2labc}, Lemma \ref{PartII2lem markov partition}. We will verify \eqref{PartII2labf} in Lemma \ref{PartII2lem condition f}.
\end{rema}
\begin{lemma}\label{PartII2lem condition f}
If we have \eqref{PartII2laba}, \eqref{PartII2labd}, \eqref{PartII2labe}, \eqref{PartII2labg}, then we have \eqref{PartII2labf}.
\end{lemma}
\begin{proof}
Let $\varepsilon$ be in \eqref{PartII2labe}.
Let $z \in \cC, q \in \N$ be given by \eqref{PartII2labd}.
We denote $\cC' = B(z, \varepsilon)$. By $\varepsilon < \varepsilon_1$, we have $\cC' \subset \cC$. We denote $J' = \T \setminus \overline{B(s_0, 2\varepsilon)}$.
We first show that
\ary \label{PartII2inf mu in ugibbs}
\inf_{ \substack{\mu : f_{*}\mu = \mu, \\ (p_1)_{*}\mu = m }} \mu(\cC' \times J') > 0
\eary
Indeed, if \eqref{PartII2inf mu in ugibbs} was false, then there would exists a sequence $\{c_n  > 0\}_{n \geq 1}$, $\{\mu_{n}, f_{*}\mu_n = \mu_n, (p_1)_{*}\mu_n = m\}_{n \geq 1}$ such that  $\lim_{n \to \infty}c_n = 0$ and $\mu_{n}(\cC' \times J') < c_n$ for all $n \geq 1$. Let $\mu$ be an accumulating point of $\mu_n$. It is clear that $\mu$ is $f-$invariant and $(p_1)_{*}\mu = m$. Moreover, we have
\aryst
\mu(\cC' \times J') \leq \liminf_{k \to \infty} \mu_k(\cC' \times J') \leq \liminf_{k \to \infty}c_k = 0
\earyst
Thus $\mu(\cC' \times J') = 0$. This implies that for $m$ almost every $z' \in \cC'$ the conditional measure $\mu_{z'}$ on $\{z'\} \times \T \simeq \T$ is supported in $\overline{B(s_0, 2\varepsilon)}$. By $z \in supp(m)$, we can let $z', z''$ be two $m$ generic points sufficiently close to $z$, such that $z'' = g^{q}(z')$, and $\mu_{z'}, \mu_{z''}$ are supported in $\overline{B(s_0, 2\varepsilon)}$. Moreover the map $D_{z',z''} : \T \to \T$ defined by $D_{z',z''}(x) = p_2 f^{q}(z', x)$, satisfies 
\aryst D_{z', z''}(\overline{B(s_0, 2\varepsilon)}) \bigcap \overline{B(s_0, 2\varepsilon)} = \emptyset 
\earyst
 By the $f-$invariance of $\mu$, for a generic choice of $z',z''$ as above, we have $D_{z',z''}\mu_{z'} = \mu_{z''}$. This is a contradiction.

 We claim that there exist arbitrarily large $K$ such that for any $\cD' \in \bA_{\varepsilon/2}(f)$, $f^{K}(\cD') \bigcap (\cC' \times J') \neq \emptyset$. Indeed,
if there was a sequence $\{ \cD_n \}_{n \geq 1} \subset \bA_{\varepsilon/2}(f)$, a strictly increasing sequence $\{ K_n \}_{n \geq 1} \subset \N$, such that for any $n \geq 1$, $f^{K_n}(\cD_n) \bigcap (\cC' \times J' ) = \emptyset$. We let $\mu$ be an accumulating point of $(f^{K_n})_{*}(Vol|_{\cD_n})$, then $\mu(\cC' \times J') = 0$. By \eqref{PartII2labg}, $\mu \in uGibbs(f)$. Then it is clear that $(p_1)_{*}\mu = m$. But then $\mu(\cC' \times J')  > 0$. Contradiction. Thus our claim is true.
 
 We let $J_1 = \T \setminus \overline{B(s_0, \varepsilon/2)}$. Then it is clear that $J' \Subset J_1$ and $d(J', J_1^c) \geq 3\varepsilon/2$. 
Take an arbitrary $\cD \in \bA_{\varepsilon}(f)$. We choose $\cD_0 \in \bA_{\varepsilon/2}(f)$ such that $\cD_0 \Subset \cD$ and $d(\cD_0, \partial \cD) > \frac{\varepsilon}{10C_2}$, where $C_2$ is in the definition of $\bA_{\varepsilon}(f)$. For any $K_0 > 0$, by our claim above there exists $K = K(K_0,\varepsilon) > K_0$, independent of the choice of $\cD, \cD_0$, such that $f^{K}(\cD_0) \bigcap (\cC' \times J') \neq \emptyset$. Let $(z',x')$ be a point in $f^{K}(\cD_0) \bigcap (\cC' \times J')$. Then by letting $K_0$ to be sufficiently large, we can find a neighbourhood of $(z',x')$ in $f^K(\cD)$, denoted by $\cD'$, such that $\cD' \in \bA_{\varepsilon}(f)$ and $diam(\cD') < \varepsilon_0$. Since $d(\cC' \times J',  \cC \times J_1 ) > 3\varepsilon/2$, we have $\overline{\cD'} \subset \cC \times J_1$. This concludes the proof.
\end{proof}

Remark \ref{PartII2rema condition a to e} and Lemma \ref{PartII2lem condition f} suggest a way of constructing dynamics satisfying condition \eqref{PartII2laba} to \eqref{PartII2labg}, as the following proposition shows.

\begin{prop}\label{PartII2thm existence of examples}
For any $\kappa \in (0,1), c_0 \in (0, 1-\kappa)$, there exists a partially hyperbolic, stably dynamically coherent, u-convergent, mostly contracting diffeomorphism $f$ on $\T^3$ satisfying \eqref{PartII2laba} to \eqref{PartII2labg}.
\end{prop}
\begin{proof}
We will follow Example (a), Section 12 in \cite{Do1}. Let $M = \T^2$ and let $g: M \to M$ be a linear Anosov diffeomorphism. It is known that the Lebesgue measure on $M$, denote by $m$, is the unique SRB measure for $g$. We let $C_0, B_0$ be projective actions of $SL(2,\R)$ on $\T$, satisfying the conditions in the beginning of this section. Let $\cC, \cB$ be two disjoint open sets of $M$ satisfying $m(\cC) > 1-\kappa, m(\cB) > c_0$ and $m(\overline{\cC} \bigcup \overline{\cB}) < 1$. We let $\cal S : M \to SL(2,\R)$ be a $C^r$ map such that $\cal S|_{\cC} \equiv C_0, \cal S|_{\cB} \equiv B_0$. We let $A: M \times \T \to \T$ be defined by $A(z,x) = \widehat{\cal S(z)}(x)$, so that \eqref{PartII2labc} is satisfied. By choosing $C_0, B_0$ to be close to rotations, it is easy to choose $\cal S$ so that \eqref{PartII2labb} is also satisfied. Since $\cC$ is an open set, and $m(\cC), m(M \setminus \overline{\cC \bigcup \cB}) > 0$, there exists a $g$ periodic point $z \in \cC \bigcap supp(\mu)$, and the $g$ orbit of $z$ intersects $M \setminus (\overline{\cC \bigcup \cB})$. We can make an arbitrarily small modification on $\cal S$ outside of $\overline{\cC \bigcup \cB}$ so that \eqref{PartII2labd} is satisfied, and the image of $\cal S$ generates $SL(2,\R)$. Moreover, any such modification will not ruin \eqref{PartII2labb}, \eqref{PartII2labc}. Now let $\varepsilon_1, \varepsilon_2$ be defined in \eqref{PartII2labd}, and let $\varepsilon = \frac{\min(\varepsilon_1, \varepsilon_2)}{20}$. Let $q > 0$ be the period of $z$, i.e. $g^{q}(z) = z$, and define $D : \T \to \T$ by $D(x) = p_2f^{q}(z,x)$. For integer $n \geq 1$, we define $f_n : M \times \T \to M \times \T$ by
\aryst
f_n(z',x') = (g^{nq+1}(z'), A(z',x')), \quad \forall (z',x') \in M \times \T
\earyst
and define $D_n : \T \to \T$ by $D_n(x) = p_2f_n^{q}(z,x)$. It is direct to verify that $D = D_n$ for all $n \geq 1$. In particular,  constant $\varepsilon_2$ is valid for all $f_n, n \geq 1$ in place of $f$.

By Remark \ref{PartII2rema condition a to e}, \eqref{PartII2laba}, \eqref{PartII2labe} are satisfied when we replace $g$ by any sufficiently large power of $g$. Since the center foliation of $g$ is a $C^1$ foliation, this is known to imply stably dynamically coherence. Moreover, by the discussion in Example (a), Section 12 \cite{Do1}, after replacing  $g$ by any sufficiently large power of $g$, $f$ become u-convergent and  mostly contracting. Then $f$ satisfifes Theorem II in \cite{Do1}, thus \eqref{PartII2labg} is verified by Corollary 6.3 in \cite{Do1}. By Lemma \ref{PartII2lem condition f}, we can replace $g$ by $g^{nq+1}$ for sufficiently large $n$, so that $f$ satisfies the conditions of Theorem II in \cite{Do1} and \eqref{PartII2laba} to \eqref{PartII2labg}.
\end{proof}

\subsection{Proving nondifferentiability}\label{PartII2Proving nondifferentiability}

The main result of this section is the following.
\begin{prop} \label{PartII2prop proving nondiff}
Let $r = 2, 3, \cdots, \infty$ and $f : M \times \T \to M \times \T$ be a $C^r$ map given by \eqref{PartII2deff} satisfying  \eqref{PartII2laba}, \eqref{PartII2labb}, \eqref{PartII2labc}, \eqref{PartII2labe}, \eqref{PartII2labf}.
 Then there exists a $C^r$ family $\{ f_t\}_{t \in (-1,1)}$ of partially hyperbolic systems through $f$, a function $\phi \in C^{r}(M \times \T, \R)$, such that for any map $t \mapsto \mu_t \in uGibbs(f_{t})$, the map $t \mapsto \int \phi d\mu_t $ is not $\beta-$H\"older at $t=0$ for any $\beta > \frac{-6\log (1-\kappa)}{\alpha}$.
\end{prop}
The following is an immediate corollary of Proposition \ref{PartII2prop proving nondiff} and Remark \ref{PartII2rema stability of conditions}.
\begin{cor}\label{PartII2stablenondiff}
For any $f: M \times \T \to M \times \T$ given in Proposition \ref{PartII2prop proving nondiff}, there exists $\sigma > 0$ such that for any $f' : M \times \T \to M \times \T$ satisfying that $d_{C^r}(f,f')  < \sigma$, $p_1 f' = p_1 f$, and that $f'(z,\cdot) \neq f(z,\cdot)$ is contained in a $\sigma-$ball, the same conclusion of Proposition \ref{PartII2prop proving nondiff} holds for $f'$ in place of $f$.
\end{cor}

\begin{proof}[Proof of Proposition \ref{PartII2prop proving nondiff}:]
By \eqref{PartII2deff}
\aryst
f(x,y) = (g(x), A(x,y))
\earyst
We let $m$ be the SRB measure of $g$, let $\varepsilon$ be given by \eqref{PartII2labe}, let $K \in \N, J_1 \Subset J \subset \T$ be given by \eqref{PartII2labf}.

We can define $f_t$ for $t \in (-1,1)$ by
\aryst
f_t(z,x) = (g(z), A(z,x)+t), \quad \forall (z,x) \in M \times \T
\earyst
Let $\hat{f} : M \times \R \to M \times \R$ be given by \eqref{PartII2labb}.
For each $t \in (-1,1)$ we define $\hat{f}_{t}: M \times \R \to M \times \R$ by
\aryst
\hat{f}_t(z,x) = (g(x), \hat{p}_2 \hat{f}(z,x) + t), \quad \forall (z,x) \in M \times \R
\earyst
It is clear that for any $t \in (-1,1)$, $\hat{p}_2\hat{f}_t(z,x+1) = \hat{p}_2\hat{f}_t(z,x) + 1, \forall (z,x) \in M \times \R$ and $\pi \hat{f}_t = f_t \pi$.
For any $z \in M$, any $x \in \T$, set
 \aryst
 \phi_t( z,x) = \hat{p}_2 \hat{f}_t(z, \hat{x}) - \hat{x}
 \earyst
 where $\hat{x}$ is any element of $\pi_{\R \to \T}^{-1}(x)$.
  The right hand side of the above equality is independent of different choices of $\hat{x}$.  
We set
  \aryst
  \phi = \phi_0.
  \earyst  

Fix any $\beta > \frac{-6\log (1-\kappa)}{\alpha}$.
We will construct a sequence of real numbers $\{t_i\}_{i \in \N}$ converging to $0$, such that for any sequence of measures $\{\mu_{i} \in uGibbs(f_{t_i}) \}_{i \in \N}$, we have
\begin{eqnarray*}
|\int \phi d\mu_{i} - \int \phi d(m \times Leb_{\T}) |>  |t_{i}|^{\beta}
\end{eqnarray*}
  
It is direct to see that $|\phi_t - \phi | \equiv |t|$ for  any $t \in (-1,1)$. Thus it suffices to show that there exists a sequence $\{t_i\}_{i \in \N}$ converging to $0$ such that for any sequence $\{ \mu_{i} \in uGibbs(f_{t_i})\}_{i \in \N}$, we have
\begin{eqnarray*}
|\int \phi_{t_i} d\mu_{i} - \int \phi d(m \times Leb_{\T})| > 2 |t_{i}|^{\beta}
\end{eqnarray*}

For any $t\in\R$, we let $R_t : \T \to \T$ be the rigid translation by $t$, i.e. $R_t(x) = x+t, \forall x \in \T$.
Since by our choice $C_0 = \widehat{H_0}$, for any $t$ sufficiently close to $0$, $C_t := R_t C_0 $ is still given by a hyperbolic element. Let $u_t, s_t$ be respectively the continuations of $u_0,s_0$. By (HYP),  there exist $c,t_1 > 0$ such that for any $t \in (-t_1,t_1)$, any $\delta \in (0,\frac{1}{2})$,
\ary \label{PartII2hyp2}
C^n_t(\T \setminus B(s_t, \delta)) \subset B(u_t, c\delta^{-1}e^{-n\alpha}), \quad \forall n \geq 1
\eary
We denote $\hat{C}_t= \hat{C} + t, \quad \hat{B}_t = \hat{B} + t$.
Then $\hat{C}_t$, $\hat{B}_t$ are  respectively lifts of $C_t$, $B_t$. We let $\hat{u}_t, \hat{s}_t$ be the fixed point of $\hat{C}_t$ which are respectively the continuations of $\hat{u}_0, \hat{s}_0$.
We have following observation. 
\begin{lemma} \label{PartII2lem strict mono fixed points} 
There exists $\gamma_1 > 0$ such that for any $t$ sufficiently close to $0$, we have
\begin{eqnarray*}
\partial_t \hat{u}_t > \gamma_1, \quad \partial_t \hat{s}_t < -\gamma_1
\end{eqnarray*}
\end{lemma}
\begin{proof}
We omit the proof for it follows from  elementary computations.
\end{proof}

We define for any $(z,x) \in M \times \T$, any $n \geq 1$, any $t \in (-1,1)$, that
\aryst
S_n(t, z,x) =  \sum_{i=0}^{n-1}\phi_{t} f_{t}^{i}(z,x) 
\earyst
We have for any $\hat{x} \in \pi_{\R \to \T}^{-1}(x)$ that
\aryst
S_n(t, z, x) = \hat{p}_2\hat{f}_t^{n}(z, \hat{x}) - \hat{x}
\earyst
By monotonicity, it is clear that 
 \ary \label{PartII2monotone S} \quad
 S_n(\delta, z,x) - S_n(-\delta,z,x) \geq 0,  \forall n \geq 0, \delta \in( 0,1), (z,x) \in M \times \T
 \eary

 For any $\cD \in \bA_{\varepsilon}(f)$, any integer $N_0 > 0$, real number $\delta \in (0,1)$, we define a sequence of random variables $X = X( \cD, N_0, \delta) = \{X_n\}_{n \geq 1}$ on probability space $(\cD, Vol|_{\cD})$, defined by
 \aryst
 X_n(z,x) = S_{nN_0}(\delta, z, x) - S_{nN_0}(-\delta, z, x), \quad \forall (z,x) \in \cD
 \earyst
 
 In the following, for any $\cD \in \bA_{\varepsilon}(f)$, any measurable subset $E \subset \cD$, any random variable $F : \cD \to \R$, we will use notations $\PP_{\cD}(E)$, $\EV_{\cD}(F)$ to denote respectively $Vol|_{\cD}(E)$ and $\int_{\cD} F dVol|_{\cD}$.
 
The following lemma is the main step in the proof.
\begin{lemma} \label{PartII2lem Z}
There exists $c_2 > 0$ such that the following is true.  For any sufficiently large integer $L > 1$, any $\delta \in [e^{-\frac{L\alpha}{3}}, \frac{1}{2})$, any $\cD_0 \in \bA_{\varepsilon}(f)$, we define a random variable $Z$ on probability space $(\cD_0, Vol|_{\cD_0})$ by
\aryst
Z(z,x) = S_{2L+1+K}(\delta, z,x) - S_{2L+1+K}(-\delta, z,x), \quad \forall (z,x) \in \cD_0 
\earyst
then we have  $Z \geq 0$ and  $
\mathbb{P}_{\cD_0}(Z \geq 1) \geq c_2(1-\kappa)^{2L}$.
\end{lemma}
\begin{proof}
By \eqref{PartII2labf} and that $J_1 \Subset J$, there exists $t_0 > 0$ such that for any $\cD_0 \in \bA_{\varepsilon}(f)$, there is a subset of $\cD_0$, denoted by $D_1$, such that for all $t \in (-t_0, t_0)$, by letting $\cD^1_t := f_{t}^{K}(D_1)$, we have $\cD^1_t \in \bA_{\varepsilon}(f)$ and $\cD^1_t \subset \cC \times J$. Moreover, let $c_1 = c_1(f,K) > 0$ such that for  any $\cD_0, \cD_1 \in \bA_{\varepsilon}(f)$ satisfying $\cD_1 \subset f^{K}(\cD_0)$, we have $c_1 Vol|_{\cD_1} \leq (f^{K})_{*}(Vol|_{\cD_0})$.

Let $\cD_0$ be given in the lemma. By \eqref{PartII2monotone S}, we have $Z(z,x) \geq 0$ for all $(z,x) \in \cD_0$. To simplify notations, we denote $\rho_0 := Vol|_{\cD_0}$.

We will inductively construct for all $0 \leq k \leq 2L$, a subset of $D_1$ denoted by $U_k$, such that $U_k \subset U_{k-1}$ for $k \geq 1$, and the following is satisfied,
\enmt
\item[$(1)$] for each $(z,x) \in U_k$, we have
\aryst
g^{K+k}(z) \in \begin{cases} \cC, \quad k \in \{0,\cdots, 2L\} \setminus \{L\}, \\ \cB, \quad k = L \end{cases}
\earyst
\item[$(2)$] there exists an integer $l_k \geq 1$, and disjoint $\cD^k_{i}\in \bA_{\varepsilon}(f), 1 \leq i \leq l_k$ such that $f^{K+k}(U_k) = \bigcup_{i=1}^{l_k} \cD^{k}_i$,
\item[$(3)$]  there exist $a_1, \cdots a_{l_k} > 0$ such that
\aryst
(f^{K+k})_{*}\rho_0 \geq \sum_{i = 1}^{l_k} a_i Vol|_{\cD^{k}_i}
, \quad  \sum_{i = 1}^{l_k} a_i  \geq \begin{cases} (1-\kappa)^{k}c_1, \quad 0 \leq k \leq L-1 \\ (1-\kappa)^{k}c_1c_0,\quad  k = L  \\ (1-\kappa)^{k}c_1c_0, \quad L < k \leq 2L  \end{cases}
\earyst
\eenmt

For $k=0$, we let $U_0 = D_1$. We denote $l_0=1$ and $\cD^{0}_1 := \cD^1_0$, then (1)-(3) are clear.

Assume that we have constructed $U_i$ for all $i \in \{0,\cdots, k\}, k \leq 2L-1$, we construct $U_{k+1}$ as follows. Let $\{\cD^{k}_i\}_{i=1}^{l_k}$ be given by (3).  Then by \eqref{PartII2labe}, for each $1 \leq i \leq l_k$,  there exist $l_{k,i} \in \N$, disjoint $\cal F^{k}_{i,j} \in \bA_{\varepsilon}(f), 1 \leq j \leq l_{k,i}$ satisfying $\cal F^{k}_{i,j} \subset f(\cD^{k}_i)$, and constants $c_{k,i, j} > 0, 1 \leq j \leq l_{k,i}$ such that
\ary \label{PartII2f k i j d k i}
\sum_{j=1}^{l_{k,i}}c_{k, i, j}Vol|_{\cal F^{k}_{i,j}} \leq f_{*}(Vol|_{\cD^{k}_i}) 
\eary
and 
\ary\label{PartII2fkij ckij}
\cal F^{k}_{i,j} \subset \begin{cases} \cC \times \T, k \neq L-1 \\ \cB \times \T, k = L-1 \end{cases}, \sum_{j=1}^{l_{k,i}} c_{k, i, j} > \begin{cases} 1-\kappa, k \neq L-1 \\ c_0, k = L-1 \end{cases}
\eary
Then we define 
\aryst
U_{k+1} = \bigcup_{i=1}^{l_k} \bigcup_{j=1}^{l_{k,i}} f^{-(K + k+1)}(\cal F^{k}_{i,j})
\earyst
It is direct to see (1),(2) for $k+1$ in place of $k$. It remains to verify (3).
By (3) for $k$ and \eqref{PartII2f k i j d k i}, we have
\aryst
(f^{K+k+1})_{*}\rho_{0} \geq \sum_{i=1}^{l_k} a_i  f_{*}(Vol|_{\cD^{k}_i})
 \geq \sum_{i=1}^{l_k} a_i \sum_{j=1}^{l_{k,i}} c_{k,i,j} Vol|_{\cal F^{k}_{i,j}}
\earyst
Then by (3) for $k$ and \eqref{PartII2fkij ckij}, we deduce (3) for $k+1$.
This concludes the induction.
 In particular, by (2),(3) for $k= 2L$, we have
\ary \label{PartII2mes u 2l+1}
\PP_{\cD_0}(U_{2L}) \geq (1-\kappa)^{2L}c_0c_1
\eary
We have the following.
\begin{lemma}\label{PartII2lem increment}
For any $(z,x) \in U_{2L}$, $Z(z,x) \geq 1$.
\end{lemma}
\begin{proof}
Without loss of generality, we can assume that $\delta > 0$ is sufficiently small, independent of $L$.
We choose an arbitrary $\hat{x} \in \pi_{\R \to \T}^{-1}(x)$.
For all $0 \leq k \leq 2L+1$, we denote
\aryst
(z_{k}, x_{k}^{\pm}) := \hat{f}_{\pm \delta}^{K+k}(z,\hat{x})
\earyst
By $(z,x) \in U_{2L} \subset D_1$, there exists $l \in \N$ such that 
\aryst
x^{\pm}_{0} \in \widehat{J} + l, \quad
z_k \in \begin{cases} \cC, \quad \forall k \in \{0,\cdots, 2L\} \setminus \{L\} \\ \cB, \quad k = L \end{cases}
\earyst
Thus we have the following relations.
\aryst
x^{\pm}_{L} = \hat{C}_{\pm \delta}^{L}(x^{\pm}_0), \quad x^{\pm}_{L+1} = \hat{B}_{\pm\delta}(x^{\pm}_{L}), \quad x^{\pm}_{2L+1} = \hat{C}_{\pm \delta}^{L}(x^{\pm}_{L+1})
\earyst
By $J \Subset \T \setminus \{s_0\}$, $\hat{s}_0 \in [0,1)$, $\hat{J} \subset [0,1)$, we have either $\widehat{J} \Subset (\hat{s}_0,\hat{s}_0+1)$ or $\widehat{J} \Subset (\hat{s}-1_0,\hat{s}_0)$. We will prove our lemma assuming the first case $\widehat{J} \Subset (\hat{s}_0,\hat{s}_0+1)$ happens. The second case is similar.

Denote
\aryst
\bar{u} = \hat{u}_{0} + l, &\quad& \bar{s} = \hat{s}_{0} + l, \\
\bar{u}_{\pm} = \hat{u}_{\pm \delta} + l ,& \quad& \bar{s}_{\pm} = \hat{s}_{\pm \delta} + l 
\earyst
Then by Lemma \ref{PartII2lem strict mono fixed points}, we have
\aryst
\bar{u}_{+} > \bar{u} + \gamma_1\delta, &\quad& \bar{s}_{+} < \bar{s} - \gamma_1 \delta \\
\bar{u}_{-} < \bar{u} - \gamma_1\delta, &\quad& \bar{s}_{-} > \bar{s} + \gamma_1 \delta 
\earyst
Then there exists $c_3 > 0$, such that for all sufficiently small $\delta > 0$, and for all $L$,
\aryst
x^{+}_L \in  (\bar{u}_{+}-c_3 e^{-L\alpha}, \bar{u}_{+} +c_3e^{-L \alpha}), \quad x^{-}_L \in (\bar{u}_{-} -c_3 e^{-L\alpha}, \bar{u}_{-}  + c_3 e^{-L\alpha})
\earyst
In particular, for sufficiently large $L$ we have
\aryst
x^{+}_L > \bar{u} + \frac{\gamma_1}{2}\delta, \quad x^{-}_{L} < \bar{u} - \frac{\gamma_1}{2} \delta
\earyst
Then
\aryst
x^{+}_{L+1} = \hat{B}(x^{+}_L) + \delta \geq \hat{B}(\bar{u}) + \delta = \bar{s} + \delta > \bar{s}_{+} + \delta \\
x^{-}_{L+1} = \hat{B}(x^{-}_L) -\delta \leq \hat{B}(\bar{u}) - \delta < \bar{s} - \delta < \bar{s}_{-} - \delta
\earyst
It is easy to see that for sufficiently large $L$, 
\aryst
x^{+}_{L+1} < \bar{s}_{+} + \frac{1}{2}, \quad x^{-}_{L+1} > \bar{s}_{-} - \frac{1}{2}
\earyst
By \eqref{PartII2hyp2}, there exists $C_4  > 0$ independent of $L$ such that\aryst
&&x^{+}_{2L+1} > \bar{u}_{+} - C_4 \delta^{-1} e^{-L\alpha} > \bar{u} + \gamma_1 \delta - C_4 \delta^{-1}e^{-L\alpha} \\
&&x^{-}_{2L+1} < \bar{u}_{-} - 1  + C_4 \delta^{-1} e^{-L\alpha} < \bar{u} - 1 - \gamma_1 \delta + C_4\delta^{-1}e^{-L\alpha}
\earyst
As a consequence,  for all sufficiently large $L$ we have
\aryst
Z(z,x) = x^{+}_{2L+1} - x^{-}_{2L+1} > 1
\earyst
\end{proof}
Now Lemma \ref{PartII2lem Z} follows from Lemma \ref{PartII2lem increment} and \eqref{PartII2mes u 2l+1}.

\end{proof}

We have the following lower bound.
\begin{lemma} \label{PartII2main lemma nondiff}
There exists a constant $c_4 > 0$ such that for all sufficiently large integer $L \geq 1$, let $\delta = e^{-\frac{L\alpha}{3}}$, then for any $\cD \in \bA_{\varepsilon}(f)$, denote $Y = \lfloor X(\cD, 2L+1+K, \delta) \rfloor$ (i.e. $Y_n = \lfloor X_n \rfloor, \forall n \geq 1 $), we have
\aryst
\liminf_{n \to \infty} \frac{1}{n}\EV_{\cD}(Y_n ) \geq c_4(1-\kappa)^{2L}
\earyst
\end{lemma}

\begin{proof}
Denote $N_0 = 2L + 1 + K$.
We have
\aryst
X_{n}(z,x) = \hat{p}_2 \hat{f}_{\delta}^{nN_0}(z, \hat{x}) - \hat{p}_2 \hat{f}_{-\delta}^{nN_0}(z, \hat{x}), \quad \forall \hat{x} \in \pi_{\R \to \T}^{-1}(x)
\earyst
We denote
\aryst
Z_n(z,x) &=& \hat{p}_2\hat{f}_{\delta}^{N_0}(\hat{f}_{\delta}^{nN_0}(z, \hat{x})) - \hat{p}_2\hat{f}_{-\delta}^{N_0}(\hat{f}_{\delta}^{nN_0}(z, \hat{x})) \\
W_n(z,x) &=& \hat{p}_2\hat{f}_{-\delta}^{N_0}(\hat{f}_{\delta}^{nN_0}(z, \hat{x})) - \hat{p}_2\hat{f}_{-\delta}^{N_0}(\hat{f}_{-\delta}^{nN_0}(z, \hat{x})) 
\earyst
Then it is clear that
\aryst
X_{n+1}  = Z_n + W_n
\earyst

By definition, $X_n \geq Y_n$. Then by the monotonicity and the periodicity of $\hat{f}_{-\delta}$, we have
\aryst
W_n(z,x) &\geq& \hat{p}_2\hat{f}_{-\delta}^{N_0}(g^{nN_0}(z), Y_n + \hat{p}_2(\hat{f}_{-\delta}^{nN_0}(z, \hat{x}))) - \hat{p}_2\hat{f}_{-\delta}^{N_0}(\hat{f}_{-\delta}^{nN_0}(z, \hat{x})) \\
&=& Y_n
\earyst
We have for all $(z,x) \in \cD$ that,
\aryst
Z_n(z, x) = S_{N_0}(\delta, f_{\delta}^{nN_0}(z,x)) - S_{N_0}(-\delta, f_{\delta}^{nN_0}(z,x) )
\earyst

By Lemma \ref{PartII2lem markov partition} and by $\varepsilon < \min(\varepsilon_1,\varepsilon_2) < \varepsilon_0(f,\frac{1}{2})$, for all $\delta$ such that $|\delta| \ll 1$, there exist $L_0 > 0$ depending only on $\varepsilon$ ( in particular, independent of  $\cD$ ), such that for all $L > L_0$, there exist disjoint $\cD_1, \cdots, \cD_k \in \bA_{\varepsilon}(f)$, satisfying that $\cD_i \subset f_{\delta}^{nN_0}(\cD)$ for all $1 \leq i \leq k$, and there exist constants $d_1, \cdots, d_k > 0$ such that
$\sum_{i=1}^{k} d_i Vol|_{\cD_i} \leq ( f_{\delta}^{nN_0})_{*}(Vol|_{\cD})$ and $\sum_{i=1}^{k}d_i  > \frac{1}{2}$.

For any $1 \leq i \leq k$, we define
\aryst
Z^{(i)}(z,x) = S_{N_0}(\delta, z, x) - S_{N_0}(-\delta, z, x), \quad \forall (z,x) \in \cD_i
\earyst
Then by Lemma \ref{PartII2lem Z} and \eqref{PartII2monotone S}, we have
\aryst
\EV_{\cD}(Y_{n+1}) - \EV_{\cD}(Y_{n})
&\geq& \PP_{\cD}( Z_n  \geq 1) \\
&\geq& \sum_{i=1}^{k}d_i\PP_{\cD_i}( Z^{(i)} \geq 1 ) \geq \frac{1}{2}(1-\kappa)^{2L} c_2
\earyst
\end{proof}
It is direct to see that for any $t \in (-1,1)$, any $\mu_t \in uGibbs(f_t)$, we have $\pi_{*}\mu_t \in uGibbs(g)$. By the uniqueness of SRB measure for $g$, we have $\pi_{*}\mu_t = m$ for all $\mu_t \in uGibbs(f_t)$. 
Then for each $t \in (0,1)$, for any $\mu_{t} \in uGibbs(f_t)$ and $\mu_{-t} \in uGibbs(f_{-t})$, there exist a subset of $M_0 \subset M$ with $m(M_0) = 1$, and for each $z \in M_0$, there exists $x,x' \in \T$ such that $(z,x)$ is $\mu_t$ generic, and $(z,x')$ is $\mu_{-t}$ generic.
Let $E \in \bA_{1}(g)$ be such that almost every $z \in E$ with respect to the Lebesgue measure on $E$ belongs to $M_0$.

We claim that : for any $y \in \T$, we have
\aryst
\lim_{n \to \infty}\frac{1}{n}\sum_{i=0}^{n-1} \phi_t f_t^{i}(z,y) = \int \phi_t d\mu_t
\earyst
Indeed, let $\hat{y}, \hat{x} \in \R$ be respectively lifts of $y,x \in \T$. Then for any $n \geq 1$, we have for $w= x,y$ that
\aryst
\frac{1}{n}\sum_{i=0}^{n-1} \phi_t f_t^{i}(z,w) = n^{-1}S_n(t, z, w) = n^{-1}(\hat{p}_2 \hat{f}^n_t(z, \hat{w}) - \hat{w})
\earyst
Since $(z,x)$ is generic for $\mu_t$, we have 
\aryst
\lim_{n \to \infty}n^{-1}(\hat{p}_2\hat{f}^n_t(z, \hat{x}) - \hat{x}) = \int \phi_t d\mu_t
\earyst
By periodicity, we have
\aryst
|(\hat{p}_2\hat{f}^n_t(z,\hat{x}) - \hat{x}) - (\hat{p}_2\hat{f}^n_t(z,\hat{y}) - \hat{y}) | \leq 2(|\hat{y} - \hat{x}|+1)
\earyst
Then the claim follows from
\aryst
\lim_{n \to \infty}\frac{1}{n}\sum_{i=0}^{n-1} \phi_t f_t^{i}(z,y)  &=& \lim_{n \to \infty}n^{-1}(\hat{p}_2\hat{f}^n_t(z, \hat{y}) - \hat{y}) \\
&=&\lim_{n \to \infty}n^{-1}(\hat{p}_2\hat{f}^n_t(z, \hat{x}) - \hat{x})
= \int \phi_t d\mu_t
\earyst
As a consequence, for any $t \in (-1,1)$, any $\mu_t \in uGibbs(f_t)$, for any $\cD \in \bA_{\varepsilon}(f)$ such that $p_1(\cD) \subset E$, we have
\aryst
\lim_{n \to \infty} \frac{1}{n}\sum_{i=0}^{n-1}\EV_{\cD}( \phi_t f_t^{i} ) = \lim_{n \to \infty} \frac{1}{n}\sum_{i=0}^{n-1}\int \phi_t f_t^{i} dVol|_{\cD} = \int \phi_t d\mu_t
\earyst
Then take an arbitrary $\cD \in \bA_{\varepsilon}(f)$ such that $p_1(\cD) \subset E$. For sufficiently large $L$ and $\delta = e^{-\frac{L \alpha}{3}}$, by Lemma \ref{PartII2main lemma nondiff}, our theorem then follows from
\aryst
 \int \phi_{\delta} d\mu_{\delta} -  \int \phi_{-\delta} d\mu_{-\delta} = \lim_{n \to \infty} \frac{1}{n(2L+1+K)}\EV_{\cD}(X_n)  \geq  \lim_{n \to \infty} \frac{1}{n(2L+1+K)}\EV_{\cD}(Y_n) 
\earyst
and that
\aryst
 \lim_{n \to \infty} \frac{1}{n(2L+1+K)}\EV_{\cD}(Y_n)  > \frac{1}{(2L+1+K)}c_4(1-\kappa)^{2L}  > 4\delta^{\beta}.
\earyst
Here $X = X(\cD, 2L+1+K, \delta)$ and $Y =  \lfloor X \rfloor$.

\end{proof}

\begin{proof}[Proof of Theorem \ref{PartII2thm nondiff}:]
By combining Proposition \ref{PartII2thm existence of examples}, Proposition \ref{PartII2prop proving nondiff}.
\end{proof}

\section*{Appendix} 
\begin{proof}[Proof of Lemma \ref{PartII2lemma ly-ineq local}:] The proof is essentially contained in  \cite{Ts} Appendix B. For the convenience of the reader, we recall the proof.

As in \cite{Ts} Appendix B, we let $\Gamma = \N \times \{+,-\}$. We let $(c(+), c(-))=(p,q)$ instead of $(1,0)$ in \cite{Ts}, and let $(c'(+), c'(-))=(p-1,q-1)$. We write $C$ for constants that does not depend on $S,\rho, \Theta, \Theta'$, while $C'$ for constants that may depend on them. Let $\mu$ be an integer such that
\aryst
2^{-\mu + 6}\norm{\zeta} \leq \norm{(DS_x)^{tr}(\zeta)} \leq 2^{\mu - 6}\norm{\zeta}, \forall x \in U, \zeta \in \R^2
\earyst
let $\nu \leq \mu - 6$ be an integer such that
\aryst
2^{\nu-1} < \Lambda(S, \Theta') \leq 2^{\nu}
\earyst
so that $\norm{(DS_x)^{tr}(\zeta)} \leq 2^{\nu}\norm{\zeta}, \forall x \in U, (DS_x)^{tr}(\zeta) \notin \CC'_{-}$.
We write as in \cite{Ts} that $(m,\tau) \hookrightarrow (n,\sigma)$ if either
\enmt
\item $(\tau, \sigma) = (+,+)$ and $m-\mu \leq n \leq \max(0, m + \nu + 6)$, or
\item $(\tau, \sigma) = \{(-,-), (+,-)\}$ and $m-\mu \leq n \leq m+\mu$.
\eenmt
and we write $(m,\tau) \not\hookrightarrow (n,\sigma)$ otherwise.

For $u \in C^r_0(R)$, let $v := Lu$. For $(n,\sigma), (m,\tau) \in \Gamma$, define
\aryst
v^{m,\tau}_{n,\sigma} = \psi_{\Theta', n, \sigma}(D) L(u_{\Theta, m, \tau})
\earyst
By Parseval's identity, we have
\ary\label{PartII2parseval}
\sum_{n,\sigma}\norm{v^{m,\tau}_{n,\sigma}}^2_{L^2} \leq C \norm{L(u_{\Theta, m,\tau})}_{L^2}^2 \leq C \gamma(S)^{-1} \norm{\rho}_{L^{\infty}}^2 \norm{u_{\Theta, m, \tau}}_{L^2}^2
\eary
We have the following.
\begin{lemma}[ Lemma B.1 in \cite{Ts} ]\label{PartII2lemmaB1inTs}
If $(m,\tau) \not\hookrightarrow (n,\sigma)$, we have
\aryst
\norm{v^{m,\tau}_{n,\sigma}}_{L^2} \leq C2^{-(r-1)\max(m,n)} \norm{u_{\Theta, m, \tau}}_{L^2}
\earyst
\end{lemma}

It is clear that 
\ary \label{PartII2normvthetaq-222} \norm{v}_{L^{2}}  \leq C \gamma(S)^{-\frac{1}{2}} \norm{\rho}_{L^{\infty}}\norm{u}_{L^2}
\eary

We claim for $q \geq 1$ that
\ary \label{PartII2normvthetaq-} \norm{v}_{H^{q}}  \leq C \gamma(S)^{-\frac{1}{2}} \norm{\rho}_{L^{\infty}}\norm{DS}^{q}\norm{u}_{\Theta, p,q} + C' \norm{u}_{\Theta, p-1, q-1} 
\eary
Indeed, for any multi-index $\alpha = (\alpha_1, \cdots, \alpha_q) \in \{1,2\}^q$, let $\partial^{\alpha}$ be as in \eqref{PartII2defpartial}, we can write
\aryst
\partial^{\alpha}v = \rho  P_0 + P_1
\earyst
where we denote $S = (S_1,S_2)$ and
\aryst
P_0 &=& \sum_{\beta = (\beta_1,\cdots, \beta_q)}\prod_{j=1}^{q}\partial_{\alpha_j}S_{\beta_j} \partial^{\beta}u \circ S \\
P_1 &=& (\partial^{\alpha}\rho) u \circ S + \sum_{\beta, 1 \leq |\beta| \leq q-1}\rho_{\beta}\partial^{\beta}u \circ S
\earyst
here $\partial^{\alpha}\rho,  \rho_{\beta}, \partial^{\alpha_j}S_{\beta_j}$ are all $C^0$ functions.

We have
\aryst
\norm{P_1}_{L^2} &\leq& C'\norm{u}_{H^{q-1}} \leq  C' \norm{u}_{\Theta, p-1, q-1}
\earyst
Moreover, we have
\aryst
\norm{\rho P_0}^2_{L^2} 
&\leq&C\norm{\rho}_{L^{\infty}}^2\norm{DS}^{2q} \gamma(S)^{-1} \sup_{\beta, |\beta| = q}\norm{\partial^{\beta}u}^2_{L^2}
\earyst
By \eqref{PartII2sobolev norm equiv} and straightforward computation,
\aryst
\norm{v}_{H^{q}} \leq C \sum_{\alpha, |\alpha| \leq q} \norm{\partial^{\alpha}v}_{L^2} \leq C\norm{\rho}_{L^{\infty}}\norm{DS}^q \gamma(S)^{-\frac{1}{2}} \norm{u}_{H^{q}}  + C' \norm{u}_{\Theta, p-1, q-1}
\earyst
Then \eqref{PartII2normvthetaq-} follows from $\norm{u}_{H^q} \leq \norm{u}_{\Theta,p,q}$ and $p \geq q$. Then our first inequality in Lemma \ref{PartII2lemma ly-ineq local} follows from \eqref{PartII2normvthetaq-} and $\norm{v}^{-}_{\Theta',q} \leq \norm{v}_{H^{q}}$.

We now prove the second inequality. 
By definition that
\ary \label{PartII2def of v+theta'}
(\norm{v}^{+}_{\Theta',p})^2 = \norm{\psi_{\Theta', 0, +}(D) v}_{L^2}^2 + \sum_{n \geq 1}2^{2pn}\norm{\psi_{\Theta', n,+}(D)v}_{L^2}^2
\eary
The first term on the right hand side of \eqref{PartII2def of v+theta'} is easily bounded by $C\norm{u}^2_{\Theta,p,q}$, or $C\norm{u}_{\Theta,p-1,q-1}^2$ if $q \geq 1$.
For any $n \geq 1$, we have
\ary\label{PartII2normpsitheta'n+} \quad
\norm{\psi_{\Theta', n, +}(D)v}_{L^2}^2 \leq 2 \norm{\sum_{(m,\tau) \hookrightarrow (n,+)} v^{m,\tau}_{n,+}} ^2 + 2\norm{\sum_{(m,\tau) \not\hookrightarrow (n,+)} v^{m,\tau}_{n,+}} ^2
\eary
Note that $(m,\tau) \hookrightarrow (n,+)$ only if $\tau = +$ and $n \leq m+\nu+6$. Thus by Cauchy's inequality
\aryst
 \norm{\sum_{(m,\tau) \hookrightarrow (n,+)} v^{m,+}_{n,+}} ^2 \leq (\sum_{l \geq n-\nu-6} 2^{-2lp})(\sum_{(m,+) \hookrightarrow (n,+)} 2^{2mp} \norm{v^{m,+}_{n,+}}_{L^2}^2)
\earyst
Then summing up the above inequality weighted by $2^{2np}$ for all $n \geq 1$, and by \eqref{PartII2parseval} we obtain
\ary
\sum_{n \geq 1} 2^{2np}\norm{\sum_{(m,\tau) \hookrightarrow (n,+)} v^{m,+}_{n,+}} ^2 &\leq& C \sum_{m} 2^{2(\nu + m) p} \gamma(S)^{-1} \norm{\rho}_{L^{\infty}}^2 \norm{u_{\Theta, m, +}}_{L^2}^2  \nonumber \\
&\leq& C \Lambda(S, \Theta')^{2p} \gamma(S)^{-1} \norm{\rho}_{L^{\infty}}^2 \norm{u}_{\Theta,p,q}^2   \label{PartII2sumngeq122np}
\eary
By Cauchy's inequality and Lemma \ref{PartII2lemmaB1inTs}, we have
\ary
&&\sum_{(n,\sigma) \in \Gamma}\norm{2^{c(\sigma)n}\sum_{(m,\tau) \not\hookrightarrow (n,+)} v^{m,\tau}_{n,+}}_{L^2}^2 \label{PartII2sumnsigmaingammanorm2} \\
 &\leq&C \sum_{(n,\sigma) \in \Gamma}(\sum_{(m,\tau)} 2^{2c(\sigma)n-2c'(\tau) m- 2(r-1) \max(m,n) } )(\sum_{(m,\tau)} 2^{2c'(\tau) m}\norm{u_{\Theta, m, \tau}}_{L^2}^2 ) \nonumber  \\
 &\leq&C \norm{u}_{\Theta, p-1,q-1}^2 \nonumber
\eary
The last inequality follows from $p,q \in [0, \frac{r}{2}-3)$ and
\aryst
2^{2c(\sigma)n-2c'(\tau) m- 2(r-1) \max(m,n) } \leq 2^{(2p-r+1)n} 2^{-(2q+r-3)m}
\earyst

We conclude the proof by \eqref{PartII2normvthetaq-222}, \eqref{PartII2normpsitheta'n+}, \eqref{PartII2sumngeq122np}, \eqref{PartII2sumnsigmaingammanorm2}.

\end{proof}

\subsection*{Acknowledgement} I thank Artur Avila for introducing to me the question on the differentiability of SRB measures and references on Banach spaces. I thank Masato Tsujii and Xin Li for discussions at ICTP. I thank Jiagang Yang for useful conversations and inputs on SRB measures and mostly contracting dynamics. 
I thank Viviane Baladi, Dmitry Dolgopyat and Stefano Galatolo for related conversations and references.
A part of this work was done while I was visiting IMPA, and I thank their hospitality.

\
\

{\footnotesize \noindent Zhiyuan Zhang\\
Institut de Math\'{e}matique de Jussieu---Paris Rive Gauche, B\^{a}timent Sophie Germain, Bureau 652\\
75205 PARIS CEDEX 13, FRANCE\\
Email address: zzzhangzhiyuan@gmail.com


\begin{thebibliography}{aaaa}


\bibitem{AvGoTs} A. Avila, S. Gou\"ezel, M. Tsujii,
\newblock {\it 
Smoothness of solenoidal attractors,} 
\newblock {Discrete and Continuous Dynamical Systems}15 (2006), no. 1, 21-35.


\bibitem{AlBoVi} J. F. Alves, C. Bonatti, M. Viana,
\newblock {\it SRB measures for partially hyperbolic systems whose central direction is mostly expanding, } 
\newblock {Invent. Math.
May 2015, Volume 140, Issue 2, pp 351-398.}




\bibitem{Ba} V. Baladi ,
\newblock {\it 
Linear response, or else,} 
\newblock {ICM Seoul.} In Volume III, 525-545 (2014).



\bibitem{BaBeSc} V. Baladi , M. Benedicks, D. Schnellmann,
\newblock {\it Whitney-Holder continuity of the SRB measure for transversal families of smooth unimodal maps,} 
\newblock {Invent. Math.
September 2015, Volume 201, Issue 3, pp 773-844.}



\bibitem{BaKuLu}  V. Baladi, T. Kuna, V. Lucarini,
\newblock {\it Linear and fractional response for the SRB measure of smooth hyperbolic attractors and discontinuous observables,} 
\newblock {Nonlinearity}, (30)  1204-1220 (2017).

\bibitem{BaTo}  V. Baladi, M. Todd,
\newblock {\it Linear response for intermittent maps,} 
\newblock {Comm. Math. Phys.}, (347) 857-874 (2016) 
 
 
 

\bibitem{BaTs} V. Baladi, M. Tsujii,
\newblock {\it Anisotropic H\"older and Sobolev spaces for hyperbolic diffeomorphisms, } 
\newblock {Ann. Inst. Fourier,} 57(2007), No. 1, 127-154.



 


\bibitem{BeYo}  M. Benedicks, L.-S. Young,
\newblock {\it Sinai-Bowen-Ruelle measure for certain H\'enon maps,} 
\newblock {Invent. Math.} 112 (1993), 541-576.




\bibitem{BoVi} C. Bonatti, M. Viana,
\newblock {\it SRB measures for partially hyperbolic systems whose central direction is mostly contracting,} 
\newblock {Israel J. Math.}   115 (2000), 157-194.



\bibitem{ChDo} N. Chernov, D. Dolgopyat,
\newblock {\it Hyperbolic billiards and statistical physics,} 
\newblock {International Congress of Mathematicians.} Vol. II, Eur. Math. Soc., Z\"urich, 2006, pp. 1679-1704.







\bibitem{DeLi} J. De Simoi, C. Liverani,
\newblock {\it Statistical properties of mostly contracting fast-slow partially hyperbolic systems,} 
\newblock {to appear in Invent. Math.} 


\bibitem{Do1} D. Dolgopyat,
\newblock {\it On dynamics of mostly contracting diffeomorphisms,} 
\newblock {Comm. Math. Physics.  213 (2000) 181-201.}


\bibitem{Do2} D. Dolgopyat,
\newblock {\it On differentiability of SRB states for partially hyperbolic systems, } 
\newblock {Invent. Math.} 155 (2004) 389--449.


\bibitem{Do3} D. Dolgopyat,
\newblock {\it On mixing properties of compact group extensions of hyperbolic systems, } 
\newblock {Israel J. Math.} 130 (2002) 157-205.



\bibitem{DoViYa} D. Dolgopyat, M. Viana, J. Yang,
\newblock {\it Geometric and measure-theoretical structures of maps with mostly contracting center,} 
\newblock {Comm. Math. Phys.}  341 (2016) 991-1014.





 
\bibitem{Ga}  S. Galatolo,
\newblock {\it Quantitative statistical stability and speed of convergence to equilibrium for partially hyperbolic skew products,} 
\newblock {arXiv. }



\bibitem{GoLi}  S. Gou\"ezel, C. Liverani,
\newblock {\it Banach spaces adapted to Anosov systems,} 
\newblock {Ergodic Theory and Dynamical Systems }26 (2006), 189-217.


\bibitem{HeHe}  H. Hennion,  L. Herv\'e,
\newblock {\it Limit theorems for Markov chains and stochastic properties of dynamical systems by quasi-compactness,} 
\newblock {Lect. Notes in Math.  }1766, (2000).





 
\bibitem{KeLi}  G. Keller, C. Liverani,
\newblock {\it Stability of the spectrum for transfer operators,} 
\newblock {Ann. Scuola. Norm. Sup. Pisa Cl. Sci.} (4) Vol. XXVIII (1999), 141-152.




\bibitem{LiSm}  A. de Lima, D. Smania,
\newblock {\it Central limit theorem for the modulus of continuity of averages of observables on transversal families of piecewise expanding unimodal maps,} 
\newblock {Journal of the Institute of Mathematics of Jussieu}, 1-61 (2016).



\bibitem{RoRoTaUr}  F. Rodriguez Hertz, M. A. Rodriguez Hertz, A. Tahzibi, R. Ures,
\newblock {\it Uniqueness of SRB Measures for Transitive Diffeomorphisms on Surfaces,} 
\newblock {Comm. Math. Phys. }306 (2011) 35-49.


\bibitem{Ru} D. Ruelle,
\newblock {\it Differentiation of SRB states,} 
\newblock { Comm. Math. Phys.}187 (1997) 227-241.




\bibitem{SiTa}  B. Simon, M. Taylor,
\newblock {\it Harmonic analysis on SL(2,R) and smoothness of the density of states in the one-dimensional Anderson model,} 
\newblock {Comm. Math. Phys.} 101 (1985) 1-19.






 
\bibitem{Ts2}  M. Tsujii,
\newblock {\it Physical measures for partially hyperbolic surface endomorphisms,} 
\newblock {Acta Math.} 194 (2005), 37-132.

 
 
\bibitem{Ts}  M. Tsujii,
\newblock {\it Decay of correlations in suspension semi-flows of angle-multiplying maps,} 
\newblock {Ergodic Theory and Dynamical Systems }28 (01), 291-317.





 
 


\end{thebibliography}
\end{document}